\documentclass[12pt]{amsart}
\usepackage{tkz-euclide}
\usepackage{latexsym}
\usepackage{amssymb,amsmath,mathrsfs}
\usepackage{amsthm}
\usepackage{hyperref}
\usepackage{enumitem}
\usepackage[capitalise]{cleveref}

\allowdisplaybreaks 
\newtheorem{theorem}{Theorem}[section]
\newtheorem{definition}{Definition}[section]
\newtheorem{prop}[theorem]{Proposition}

\newtheorem{lemma}[theorem]{Lemma}\newtheorem{innercustomthm}{Proposition}
\newenvironment{customthm}[1]
  {\renewcommand\theinnercustomthm{#1}\innercustomthm}
  {\endinnercustomthm}

\newcommand{\Bdot}{\mathring{B}}

\newcommand{\Hdot}{\mathring{H}}
\newcommand{\Cdot}{\mathring{C}}
\newcommand{\diam}{\mathrm{diam}}
\newcommand{\Lip}{\operatorname{Lip}}
\newcommand{\LL}{\operatorname{LL}}
\newcommand{\BMO}{\operatorname{BMO}}
\newcommand{\QG}{\operatorname{QG}}
\newcommand{\half}{(-\overline{\Delta})^\frac{1}{2}}
\newcommand{\quarter}{(-\overline{\Delta})^\frac{1}{4}}
\newcommand{\fractional}{(-\overline{\Delta})^\alpha}
\newcommand{\eighth}{(-\overline{\Delta})^\frac{1}{8}}

\newcommand{\grad}{\overline{\nabla}}
\newcommand{\lap}{\overline{\Delta}}

\newcommand{\Dbar}{\overline{D}}

\newcommand{\dnu}{\partial_\nu \Psi}
\newcommand{\logplus}{\log_+}

\def\Xint#1{\mathchoice
{\XXint\displaystyle\textstyle{#1}}%
{\XXint\textstyle\scriptstyle{#1}}%
{\XXint\scriptstyle\scriptscriptstyle{#1}}%
{\XXint\scriptscriptstyle\scriptscriptstyle{#1}}%
\!\int}
\def\XXint#1#2#3{{\setbox0=\hbox{$#1{#2#3}{\int}$ }
\vcenter{\hbox{$#2#3$ }}\kern-.6\wd0}}

\def\dashint{\Xint-}

\newcommand{\dive}{\operatorname{div}}
\newcommand{\curl}{\operatorname{curl}}
\newcommand{\dist}{\mathrm{dist}}
\newcommand{\supp}{\operatorname{supp}}

\numberwithin{equation}{section}

\title[Global in Time Classical Solutions to the 3D Quasi-geostrophic System for Large Initial Data]{Global in Time Classical Solutions to the 3D Quasi-geostrophic System for Large Initial Data}

\author[Novack]{Matthew D. Novack}
\address[Matthew D. Novack]{\newline Department of Mathematics, \newline The University of Texas at Austin, Austin, TX 78712, USA}
\email{mnovack@math.utexas.edu}

\author[Vasseur]{Alexis F. Vasseur}
\address[Alexis F. Vasseur]{\newline Department of Mathematics, \newline The University of Texas at Austin, Austin, TX 78712, USA}
\email{vasseur@math.utexas.edu}

\date{\today}


\thanks{\textbf{Acknowledgment.} A. Vasseur was partially supported by the NSF Grant DMS 1209420.}

\begin{document}

\begin{abstract} 
In this paper, the authors show the existence of global in time classical
solutions to the 3D quasi-geostrophic system with Ekman pumping for any
smooth initial value (possibly large). This system couples an inviscid transport equation in $\mathbb{R}^3_+$ with an equation on the boundary satisfied by the trace. The proof combines the De Giorgi regularization effect on the boundary $z=0$ --similar to the so called surface quasi-geostrophic equation-- with Beale-Kato-Majda techniques to propagate regularity for $z>0$.
A bootstrapping argument combining potential theory and Littlewood-Paley techniques is used to strengthen the regularization effect on the trace up to the Besov space $\Bdot_{\infty,\infty}^1$.

\end{abstract}

\maketitle \centerline{\date}

\section{Introduction}
We consider the 3D quasi-geostrophic system $(\QG)$, which can be stated as the following set of equations imposed upon the stream function $\Psi: [0,\infty) \times \mathbb{R}_+^3 \rightarrow \mathbb{R}$:
\[
\left\{
       \begin{array}{@{}l@{\thinspace}l}
       \partial_t(\Delta \Psi) + \grad^\perp \Psi \cdot \grad (\Delta \Psi) = 0  \hspace{.89in}  t>0,\hspace{.1in} z>0,\hspace{.1in} x=(x_1,x_2)\in\mathbb{R}^2  \\
       \partial_t(\dnu)+ \grad^\perp \Psi \cdot \grad(\dnu)= \lap \Psi \hspace{.67in}  t>0,\hspace{.1in} z=0,\hspace{.1in} x=(x_1,x_2)\in\mathbb{R}^2 \hspace{.3in} \text{(QG)}\\
       \Psi(0,z,x) = \Psi_0(z,x)  \hspace{1.485in}  t=0,\hspace{.1in} z\geq0,\hspace{.1in} x=(x_1,x_2)\in\mathbb{R}^2
       \end{array}  \right.
.\]
As a convention, we choose the vertical component to be the first component of any vector in $\mathbb{R}_+^3$. We employ the following notation:
$$ \grad \Psi = (0, \partial_{x_1} \Psi, \partial_{x_2} \Psi), $$ and $$ \lap \Psi = \partial_{x_1x_1} \Psi + \partial_{x_2x_2} \Psi .$$
The velocity field for the stratified flow is given by $$ \grad^\perp \Psi = (0, -\partial_{x_2} \Psi, \partial_{x_1} \Psi).$$
At the boundary $z=0$, $\dnu$ is a function of $x$ and $t$ only and denotes the Neumann condition $$ \dnu(t,x) = -\partial_z\Psi(t,0,x).$$  Here $\Delta\Psi = \partial_{zz} \Psi + \partial_{x_1x_1}\Psi + \partial_{x_2x_2}\Psi$ is the usual Laplacian.  In this article, we construct a unique, global smooth solution of (QG) for any initial value $\Psi_0$ which is smooth enough.  

The 3D quasi-geostrophic system is a widely used model in oceanography and meteorology to describe large-scale oceanic and atmospheric circulation. The system includes two coupled equations. First, beginning with Navier-Stokes and accounting for the rotation of the Earth, one derives a transport equation on the vorticity.  Second, a careful analysis of the Ekman layers near the boundary produces an equation which $\dnu$ satisfies. Chemin \cite{cg} considered the convergence in the limit of solutions to the primitive equations to a solution of the quasi-geostrophic equation. In addition, rigorous derivations were carried out by Beale and Bourgeois \cite{bb} in the absence of the boundary layer and Desjardins and Grenier \cite{dg} with the inclusion of the boundary layer.   Much of the difficulty in the analysis in fact stems from the boundary layer.  Taking advantage of the viscous term on the boundary, Desjardins and Grenier \cite{dg} constructed global weak solutions.  Recently, global weak solutions were constructed in the inviscid case \cite{pv}.  Much recent work has also been focused on a simplified model first studied by Constantin, Majda, and Tabak \cite{cmt} and known as the surface quasigeostrophic equation (SQG).  There are different variants of SQG depending on the strength of the diffusive term.  In the critical case, global regularity has been obtained by several different authors, each utilizing different techniques; see Kiselev, Nazarov, and Volberg \cite{knv}, \cite{cv}, Kiselev and Nazarov \cite{kn}, and Constantin and Vicol \cite{cvicol}.   Many authors have also emphasized the connection between critical SQG and 3D Navier-Stokes and have used versions of SQG, especially the inviscid one, as toy models for 3D fluid equations (see Constantin \cite{Constantin2006} and Held, Garner, Pierrehumbert, and Swanson \cite{MR1312238}).  

This paper is dedicated to a proof of the following well-posedness result for $(\QG)$.  
\begin{theorem}\label{main}
Let the initial data $\nabla\Psi_0 \in H^s(\mathbb{R}_+^3)$ for some $s\geq3$.  Then there exists a unique classical solution $\Psi$ to $(\QG)$ satisfying the following: for all $T>0$, there exists $C(T,s)$ such that for all $t\leq T$, $||\nabla\Psi(t, \cdot)||_{H^s(\mathbb{R}_+^3)} \leq C(T,s)$.  In addition, if the initial data $\nabla\Psi_0 \in H^s(\mathbb{R}_+^3)$ for all $s$, then for all $T$, $\Psi \in C^\infty([0,T]\times \mathbb{R}_+^3)$. 
\end{theorem}
  The bulk of the proof is centered around verifying a version of the Beale-Kato-Majda criterion from \cite{bkm}.  The idea is to first decompose the solution $\Psi = \Psi_1 + \Psi_2$ into two components as follows: 

\[
 \left\{
       \begin{array}{@{}l@{\thinspace}l}
       \Delta \Psi_1 = 0 \\
       \dnu_1 = \dnu 
       \end{array}  \right.
\hspace{.3in} 
 \left\{
       \begin{array}{@{}l@{\thinspace}l}
       \Delta \Psi_2 = \Delta \Psi \\
       \dnu_2 = 0.
       \end{array}  \right.
\]

  Intuitively, $\Psi_1$ is the problematic term since it contains the boundary condition.  We will find that $\partial_\nu \Psi_1$ satisfies an equation resembling critical 2D SQG, with an adjustment to the drift term and a forcing term appearing from the presence of $\Psi_2$. To show that $\dnu_1$ is H\"{o}lder continuous, we utilize the De Giorgi technique following \cite{cv} and \cite{v} (see also Friedlander and Vicol \cite{Friedlander2011} for an application to active scalar equations).   We then improve the regularity using Littlewood-Paley techniques and potential theory to bootstrap (see \cite{cv}, Constantin and Wu \cite{cw}, and Dong and Pavlovi\'{c} \cite{pavlovic}).   In order to then show global well-posedness, one generally requires Lipschitz regularity or a suitable substitute on the velocity $\grad^\perp \Psi$.  Due to the fact that $\partial_x \Psi_1, \partial_y \Psi_1$ are related to $\partial_z \Psi_1$ via the Riesz transforms and the fact that $\grad \Psi_2$ is not even Lipschitz, the Besov version of Lipschitz regularity must suffice.  In the literature, this space is referred to as $\Bdot_{\infty, \infty}^1$, or the Zygmund class. The texts of Stein \cite{Stein} and Grafakos \cite{Grafakos} include thorough expositions of the essential theory, while Chemin \cite{Chemin} and Bahouri, Chemin, and Danchin \cite{bcd} have detailed the application of the Zygmund class to the study of wide classes of PDE's, particularly the incompressible Euler equations.  For us, the most useful property of Besov spaces will be an inequality which controls the $L^\infty$ norm by the $\Bdot_{\infty,\infty}^0$ Besov norm, a lower Sobolev norm, and a logarithm of a higher Sobolev norm.  From there, we can prove propagation of regularity.  

The first section of the paper sets the notation and recalls some necessary results.  The second section contains the proof of the $C^\alpha$ regularity on $\dnu_1$.  In the third section, we bootstrap the regularity of $\dnu_1$ (and therefore $\nabla \Psi_1$) up to $\Bdot_{\infty, \infty}^1$.  In the last section, we show the propagation of regularity.  The appendix provides sketches of several calculations, some of which can be found in Bahouri, Chemin, and Danchin \cite{bcd}, Chemin \cite{Chemin}.  We record them here for the sake of completeness and readability.

\section{Notation and Preliminaries}

We use the notation $L^p(\mathbb{R}^n)$ for the Lebesgue spaces.  We denote the usual Hilbert Sobolev spaces (for fractional and integer $s$) by $H^s(\mathbb{R}^n)$.  The homogeneous Sobolev spaces are denoted $\Hdot^s(\mathbb{R}^n)$ and are defined as the space of functions $f$ such that $(-\Delta)^\frac{s}{2} f \in L^2$.  Equivalently, we can define $\Hdot^s(\mathbb{R}^n)$ for $s\in(0,1)$ using the Gagliardo seminorm (see Di Nezza, Palatucci, and Valdinoci \cite{hitchhiker}).  To define $\Hdot^\frac{1}{2}(\Omega)$ for bounded sets $\Omega \in \mathbb{R}^n$, we shall use the Gagliardo seminorm.  Negative Sobolev spaces $H^{-z}(\Omega)$ or $ H^{-z}(\mathbb{R}^n)$ for $z\in \mathbb{N}$ are defined as the duals of $H^z_0(\Omega)$ or $H^z(\mathbb{R}^n)$, respectively.   We use the notation $\nabla^s f$ to denote the collection of all partial derivatives of order $s\in\mathbb{N}$ .

In this paper, we consider functions defined on $\mathbb{R}^2$ or $\mathbb{R}_+^3 = [0,\infty)\times\mathbb{R}^2$ .  It will be convenient to keep track of when functions are being differentiated in $x$ only.  For that reason, and also to emphasize when we are considering functions defined on $\mathbb{R}^2$, we employ the following notations.  

\begin{definition} Let $f$ be a real-valued function defined on $\mathbb{R}_+^3$.  Put $\lap f = \partial_{x_1x_1}f+\partial_{x_2x_2}f$ and $\grad f = (0, \partial_{x_1} f, \partial_{x_2} f)$.  Let $(\fractional f)\string^(z,\xi)=\hat{f}(z,\xi)\cdot|\xi|^{2\alpha} $, where the Fourier transform is being taken in $x$ only for each fixed $z$ (ignoring constants coming from the Fourier transform).  For a partial differential operator with multi-index $\alpha=(\alpha_1, \alpha_2)$, $\Dbar^\alpha f$ denotes differentation in the flat variables $(x_1, x_2)$.  When $f$ is only defined on $\mathbb{R}^2$, we will use the above symbols to denote the usual differential operators.   
\end{definition}

We recall the well known fact that the characteristic function $\mathcal{X}_E$ of a bounded, Lebesgue measurable set $E$ belongs to $H^s$ if and only if $s<\frac{1}{2}$ (see Bourgain, Brezis, and Mironescu \cite{Bourgain2000} for a detailed discussion).  The following is a corollary which will be necessary to prove the decrease in oscillation in the De Giorgi argument.  

\begin{prop}\label{bbm}
Let $\phi$ be a radially symmetric and decreasing, $C^\infty$ bump function such that $0\leq\phi(x) \leq 1$ for all $x$, $\phi=1$ on $B_1(0)$, and $\supp\phi\subset B_2(0)$. Let $r(x)$ be a nonnegative, bounded function such that $r^2(x) \in H^\frac{1}{2}(\supp\phi)$.  Then if $\{ x: 0 < r^2(x) < \frac{1}{4} \phi^2 \}$ is empty, either $r=0$ or $r^2 \geq \frac{1}{4}\phi^2$ on $\supp \phi$.
\end{prop}

Lipschitz spaces and their variants will be referred to frequently throughout.  

\begin{definition}
\begin{enumerate}
\item For $\alpha\in(0,1)$, let $C^\alpha = \{ f:||f||_{C^\alpha}<\infty \}$, where 
$$||f||_{C^\alpha} = ||f||_{L^\infty} + \sup_{x\neq y}{\frac{|f(x)-f(y)|}{|x-y|^\alpha}}.$$
Also, the homogeneous space $\Cdot^\alpha$ is defined as 
$$ \{ f: \sup_{x\neq y}{\frac{|f(x)-f(y)|}{|x-y|^\alpha } }<\infty \} .$$
\item Let $\Lip = \{ f:||f||_{\Lip}<\infty \}$, where 
$$||f||_{\Lip} = ||f||_{L^\infty} + \sup_{x\neq y}{\frac{|f(x)-f(y)|}{|x-y|}}.$$
\item Let the space of log-Lipschitz functions $\LL = \{ f:||f||_{\LL}<\infty \}$, where 
$$||f||_{\LL} = ||f||_{L^\infty} + \sup_{|x-y|<1, x\neq y}{\frac{|f(x)-f(y)|}{|x-y|(1-\log(|x-y|))}}.$$
\end{enumerate}
\end{definition}

	Let us now recall the classical Littlewood-Paley operators and the relevant function spaces, as well as some equivalences.  Let $\mathcal{S}(\mathbb{R}^n)$ denote the  Schwarz class of rapidly decaying smooth functions, and $\mathcal{S}'(\mathbb{R}^n)$ the dual space of tempered distributions.  Letting $\mathcal{P}$ denote the space of polynomials, we construct the space $\mathcal{S}' / \mathcal{P}$, i.e., tempered distributions modulo polynomials.  
	We employ the standard dyadic decomposition of $\mathbb{R}^n$, specifically a sequence of smooth functions $\{\Phi_j \}_{j\in \mathbb{Z}}$ such that 
$$\supp \hat{\Phi}_j \subset \{\xi\in\mathbb{R}^n : |\xi|\in (2^{j-1},2^{j+1}) \}  $$	
and
\[
\sum_{j\in\mathbb{Z}}{\hat{\Phi}_j(\xi)}= \left\{
       \begin{array}{@{}l@{\thinspace}l}
       0 \hspace{.2in} \textup{if} \hspace{.2in} \xi=0\\
       1 \hspace{.2in} \textup{if} \hspace{.2in} \xi\in\mathbb{R}^n \setminus \{0\}.\
       \end{array}  \right.
\]

For $f\in \mathcal{S}' / \mathcal{P}$ and $j \in \mathbb{Z}$, we define $\Delta_j f = \Phi_j \ast f$.

\begin{definition} 
For $s\in\mathbb{R}$ and $1\leq p,q \leq \infty$, the space $\Bdot_{p,q}^s$ is defined as $$ \{ f \in \mathcal{S}' / \mathcal{P}: ||f||_{\Bdot_{p,q}^s}<\infty \} $$
where the homogeneous Besov norm is defined as the $l^p$ norm of the doubly-infinite sequence of Littlewood-Paley projections:
$$ ||f||_{\Bdot_{p,q}^s} = ||\{2^{js}||\Delta_j f||_{L^q}\}_{j\in\mathbb{Z}}||_{l^p}.$$
\end{definition}
In nearly every usage throughout the paper, the Littlewood-Paley projections and the accompanying Besov norms are in $x=(x_1,x_2)$ only; for clarity and emphasis we will use the notation $\Bdot_{p,q}^s(\mathbb{R}^2)$.

We record the following Bernstein inequalities (see \cite{bcd}).  
\begin{prop}\label{bernstein}
\begin{enumerate}
\item Let $\mathcal{C}$ be an annulus in $\mathbb{R}^d$, $m\in \mathbb{R}$, and $k=2\lfloor 1+ \frac{d}{2}\rfloor$.  Let $\sigma$ be a k-times differentiable function on $\mathbb{R}^d \setminus\{0\}$ such that for any $\alpha \in \mathbb{N}^d$ with $|\alpha| \leq k$, there exists a constant $C_\alpha$ such that 
$$ \forall \xi \in \mathbb{R}^d, |D^\alpha \sigma(\xi)| \leq C_\alpha |\xi|^{m-|\alpha|}.$$
There exists a constant $C$, depending only on the constants $C_\alpha$, such that for any $p\in[1,\infty]$ and any $\lambda>0$, we have, for any function $u$ in $L^p$ with Fourier transform supported in $\lambda\mathcal{C}$, 
$$ ||(\sigma(\xi)\hat{u}(\xi))^{\vee}||_{L^p} \leq C\lambda^m||u||_{L^p} .$$
\item Let $p\in[1,\infty]$ and $s\in\mathbb{R}$.  Then for any $j\in\mathbb{Z}$, there exist constants $c_1, c_2$ such that 
$$ c_1 2^{2j\alpha} ||\Delta_j u ||_{L^p} \leq || (-\Delta)^\alpha \Delta_j u||_{L^p} \leq c_2 2^{2j\alpha} ||\Delta_j u||_{L^p}  $$
\end{enumerate}
\end{prop}
We several corollaries in the following proposition.  
\begin{prop}\label{corollaries} Let $s\in \mathbb{R}$, $p,q \in [1,\infty]$.  
\begin{enumerate}
\item Let $\mathcal{R}_j$ denote the $j^{th}$ Riesz transform with Fourier multiplier $\frac{i\xi_j}{|\xi|}$.  Then $\mathcal{R}_j$ is a bounded linear operator from $\Bdot_{p,q}^s$ to itself. 
\item Let $\alpha$ be a multi-index. Then the partial differential operator $D^\alpha$ is bounded from $\Bdot_{p,q}^s$ to $\Bdot_{p,q}^{s-|\alpha|}$.
\item Given $\alpha\in\mathbb{R}$, the operator $(-\Delta)^\alpha$ is bounded from $\Bdot_{p,q}^s$ to $\Bdot_{p,q}^{s-2\alpha}$.  
\item For $\alpha\in \mathbb{R}$ and $p\in[1,\infty]$, $||f||_{\Bdot_{p,\infty}^\alpha} \leq || (-\Delta)^\alpha f ||_{L^p}$.
\end{enumerate}
\end{prop}
We collect several facts concerning the Besov spaces $\Bdot_{\infty,\infty}^s$.  For a more detailed discussion as well as proofs, see \cite{Grafakos}.  
\begin{prop}\label{equivalences}
\begin{enumerate}
\item The space $\Bdot_{\infty,\infty}^1$ can be characterized as the space of functions such that
$$ ||f||_{\Bdot_{\infty,\infty}^1} = \sup_{x,y\in\mathbb{R}^n, y\neq 0}{\frac{|f(x+y)+f(x-y)-2f(x)|}{|y|}<\infty} $$
with equivalence in norm holding between the difference quotient and Littlewood-Paley characterizations.  
\item For non-integer values of $s$, the spaces $\Bdot_{\infty,\infty}^s$ and $\Cdot^s$ are equivalent, with an equivalence in norm (which is not uniform in $s$).  
\item For any strictly positive $s$, the restriction of any function $f\in\Bdot_{\infty,\infty}^s(\mathbb{R}^n)$ to any $k$-dimensional affine subset produces a function in $\Bdot_{\infty,\infty}^s(\mathbb{R}^k)$ with $||f||_{\Bdot_{\infty,\infty}^s(\mathbb{R}^k)} \leq ||f||_{\Bdot_{\infty,\infty}^s(\mathbb{R}^n)}$.
\end{enumerate}
\end{prop}
 
The following proposition will be used briefly in the isoperimetric lemma in the De Giorgi argument.  We prove it in the appendix.  

\begin{prop}\label{duality}
\begin{enumerate}
\item Suppose that $(-\lap)^{-\frac{1}{4}} w \in L^\infty(\mathbb{R}^2)$ and $z \in \Hdot^{\frac{1}{2}}\cap L^\infty(\mathbb{R}^2)$ is supported in $B_2(0)$.  Then there exists $C$ independent of $w,z$ such that $$|| wz ||_{H^{-2}(\mathbb{R}^2)} \leq C ||(-\lap)^{-\frac{1}{4}} w||_{L^\infty(\mathbb{R}^2)}  \left( ||z||_{L^\infty(\mathbb{R}^2)} + ||z||_{\Hdot^\frac{1}{2}(\mathbb{R}^2)} \right) $$
\item Suppose that $z \in L^\infty \cap \Hdot^\frac{1}{2}(\mathbb{R}^2)$.  Then there exists $C$ independent of $z$ such that $$||z \half z||_{H^{-2}(\mathbb{R}^2)} \leq C\left( ||z||_{L^\infty(\mathbb{R}^2)}||z||_{\Hdot^\frac{1}{2}(\mathbb{R}^2)}+ ||z||^2_{\Hdot^\frac{1}{2}(\mathbb{R}^2)}\right)$$. 
\end{enumerate}
\end{prop}
 
We shall need to control the $L^\infty$ norm of a function by the $\Bdot_{\infty,\infty}^0$ Besov norm and some Sobolev norms.  The following inequality will suit our purposes; the proof follows that of Proposition 2.104 in \cite{bcd}, and we include it in the appendix.  See also \cite{Chemin} for the same result.  

\begin{prop}\label{inequality}
There exists a constant $C$ such that for any $h = \grad H:\mathbb{R}^2\rightarrow\mathbb{R}^2$,
$$ ||h||_{L^\infty} \leq C||H||_{L^\infty} + {C}||h||_{\Bdot_{\infty,\infty}^0}\left( 1+\log{\frac{||h||_{\Hdot^\frac{3}{2}}}{||h||_{\Bdot_{\infty,\infty}^0}}} \right).$$
\end{prop}
	
In order to prove propagation of regularity, we shall use the classical commutator estimate whose proof may be found in Klainerman and Majda \cite{km}. In our case, the control of $||\nabla f||_{L^\infty}, ||g||_{L^\infty}$ will come from the Besov regularity of $f$ and $g$ and \cref{inequality}.
\begin{prop}\label{commutator}
Assume $f,g \in H^s(\mathbb{R}^n)$. Then for any multi-index $\alpha$ with $|\alpha|=s$, we have 
$$ ||D^\alpha (fg)-f D^\alpha g ||_{L^2} \leq C(s) \left(||\nabla f||_{L^\infty}||\nabla^{(s-1)}g||_{L^2} + ||g||_{L^\infty}||\nabla^s f||_{L^2}\right). $$
\end{prop}

We will require the following lemmas concerning $\BMO$ functions to carry out the De Giorgi argument. Here we use $\BMO$ to refer to the space of functions with bounded mean oscillation equipped with the usual norm.  The first two lemmas are well-known properties of functions belonging to $\BMO$ (see \cite{Grafakos}).  The third follows from the John-Nirenberg inequality.  The fourth follows from the third in conjunction with a generalization of the Cauchy-Lipschitz theorem for $L^1(\LL)$ vector fields (see Theorem 3.7 in Chapter 3 of \cite{bcd}).   Integrals with a dash through the center are average values.  

\begin{prop}\label{bmofacts}
\begin{enumerate}
\item Let $Q$ denote any cube in $\mathbb{R}^n$.  For all $0<p<\infty$, there exists a finite constant $B_{p,n}$ such that 
$$ \sup_{Q} \left( \dashint_Q \left|f - \dashint_Q f\right|^p \right)^\frac{1}{p} \leq B_{p,n}||f||_{\BMO}.$$
\item Let $B_1$, $B_2$ be two balls in $\mathbb{R}^n$ such that there exists $A$ such that $$A^{-1}\diam (B_2) \leq \diam (B_1) \leq A \hspace{.03in}\diam (B_2)$$
and
$$ \dist(B_1,B_2) \leq A \hspace{.03in} \diam (B_1) .$$
Then there exists a constant $C(A)$ such that for any $u\in\BMO$
$$ \left|\dashint_{B_1}u-\dashint_{B_2}u\right| \leq C(A)||u||_{\BMO} .$$
\item Let $u\in \BMO$ and satisfy
$$ \sup_{x \in \mathbb{R}^n} \dashint_{B_1(x)} u(y)\,dy < \infty $$
and define 
$$ f(x) = \dashint_{B_1(x)}u(y)\,dy .$$
Then $f(x)$ is log-Lipschitz $(\LL)$ in $x$.
\item Let $u(t,x): [-2,0]\times \mathbb{R}^2 \rightarrow \mathbb{R}$ belong to $ L^\infty([-2,0];\BMO(\mathbb{R}^2)) \cap L^\infty([-2,0];L^2(\mathbb{R}^2)) $.  Then the following ordinary differential equation has a unique Lipschitz solution which satisfies the ODE almost everywhere in time.
$$  \left\{
       \begin{array}{@{}l@{\thinspace}l}
       \dot{\Gamma}(t) = \dashint_{B_1(\Gamma(t))} u(t,y)\,dy\\
       \Gamma(0) = 0\
       \end{array}  \right.  $$
\end{enumerate}
\end{prop}

We shall make use of the following well-known trace estimate for Sobolev functions.  

\begin{lemma}\label{trace}
Suppose that $\nabla u\in L^2(\mathbb{R}_+^3)$.  Then $u|_{z=z_0}\in \Hdot^\frac{1}{2}(\mathbb{R}^2)$ with the trace estimate $||u(z_0,\cdot)||_{\Hdot^\frac{1}{2}(\mathbb{R}^2)} \leq ||\nabla u||_{L^2(\mathbb{R}_+^3)}$. 
\end{lemma}

In \cite{pv}, the authors prove the existence of weak solutions to the inviscid quasi-geostrophic system.  The proof reformulates the system into a transport equation on $\nabla \Psi$ and relies on the following orthogonal decomposition of $L^2$ vector fields.  Given an $L^2$ vector field $u$, we can decompose $u$ as $u=\mathbb{P}_\nabla u + \mathbb{P}_{\curl}u$.  Here $\mathbb{P}_\nabla u = \nabla v$ for some scalar function $v$.  Furthermore, $\mathbb{P}_{\curl}u=\curl(w)$ for some $L^2$ vector field $w$ with $\mathbb{P}_{\curl}u \cdot \nu =0$.  If $u$ is smooth enough to define a trace, then $ u \cdot \nu =  \mathbb{P}_\nabla u \cdot \nu$. Note that the operator $\mathbb{P}_{\nabla}$ commutes with $\Dbar^\alpha$ but not $D^\alpha$.

\begin{prop}\label{Pgradient}
Let $\Psi$ be a smooth solution to $(\QG)$.  Then if $F$ solves the Neumann problem with $\Delta F=0$ and $\partial_\nu F = \lap \Psi|_{z=0}$, $\nabla \Psi$ satisfies the following equation, which we shall refer to as $(QG_\nabla)$:

$$ \partial_t(\nabla \Psi)+ \mathbb{P}_\nabla(\grad^\perp \Psi \cdot \grad(\nabla \Psi)) = \nabla F .$$ 
\end{prop}

For $\Psi$ a smooth solution to $(QG_\nabla)$, taking both the divergence and the trace shows that $\Psi$ also solves $(QG)$.  Finally, we state a local existence theorem and the necessary \textit{a priori} estimates.  For a proof of the following local existence theorem, one can employ the standard semigroup approach found in Kato \cite{Kato}.

\begin{prop}\label{localexistence}
For any initial data $\nabla\Psi_0 \in H^3(\mathbb{R}_+^3)$ for $(\QG)$, there exists a time interval $[0,\bar{T}]$, where $\bar{T}$ depends only on the size of $||\nabla\Psi_0||_{H^3(\mathbb{R}_+^3)}$, such that $(\QG)$ has a solution $\nabla \Psi \in L^\infty([0,T];H^3(\mathbb{R}_+^3))$.  
\end{prop}

The following proposition contains the a priori estimates which we shall use to prove global existence.  Each estimate depends only on the size of the norm of the initial data $||\nabla \Psi_0||_{H^3(\mathbb{R}_+^3)}$.  The strategy of the proof will be to obtain a differential equality on the norm $||\nabla\Psi(t)||_{H^3(\mathbb{R}_+^3)}$ which depends only on the a priori regularity.  Then we apply a continuation principle.  Specifically, assume $(QG)$ has a smooth solution $\nabla\Psi\in L^\infty([0,T-\epsilon];H^3(\mathbb{R}^3_+))$ for every $\epsilon>0$, but that the solution cannot be continued beyond time $T$. The differential inequality will show that in fact $||\nabla\Psi||_{H^3(\mathbb{R}_+^3)}$ is uniformly bounded on $[0,T]$.  Then applying the local existence theorem shows that the solution can be continued past $T$.  In this way, we show that the system admits a unique classical solution for all time.  

\begin{prop}\label{apriori}  Let $\nabla\Psi \in L^\infty([0,T];H^3(\mathbb{R}_+^3))$ be a smooth solution to (QG) on the interval $[0,T]$.  Then there exists a universal $C$ independent of $T$ and $\Psi$ such that $\Psi$ satisfies the following for all $t\in[0,T]$:
\begin{enumerate}
\item $\frac{1}{2} ||\nabla \Psi(t)||_{L^2(\mathbb{R}_+^3)}^2 + ||\grad \Psi(t)|_{z=0} ||_{L^2([0,t];L^2(\mathbb{R}^2))}^2 \leq  || \nabla\Psi_0 ||_{L^2(\mathbb{R}_+^3)}^2$.
\item For all $p\in[2,\infty]$, $||\Delta\Psi(t)||_{L^p(\mathbb{R}_+^3)} = ||\Delta\Psi_0||_{L^p(\mathbb{R}_+^3)} \leq C ||\nabla\Psi_0||_{H^3(\mathbb{R}_+^3)}$.
\item $ ||(-\lap)^\frac{3}{4} \Psi_2(t) |_{z=0} ||_{L^2(\mathbb{R}^2)} \leq  ||\Delta\Psi_0||_{L^2(\mathbb{R}_+^3)} \leq C ||\nabla\Psi_0||_{H^3(\mathbb{R}_+^3)}$.
\item For $z=z_0\geq 0$, $||\nabla\Psi_2(t)|_{z=z_0}||_{\Bdot_{\infty,\infty}^1(\mathbb{R}^2)} \leq || \Delta\Psi_0 ||_{L^\infty(\mathbb{R}_+^3)} \leq C|| \nabla\Psi_0 ||_{H^3(\mathbb{R}_+^3)}$.
\item $||(-\lap)^\frac{3}{4} \Psi_2(t) |_{z=0}||_{C^\frac{1}{2}(\mathbb{R}^2)} \leq C ||\nabla\Psi_0||_{H^3(\mathbb{R}_+^3)}$.
\item $||\dnu(t)||_{L^2(\mathbb{R}^2)} \leq C|| \nabla\Psi_0 ||_{H^3(\mathbb{R}_+^3)}(1+t) $.
\item For $p\in[4,\infty]$ and $z_0\geq 0$, $||\nabla\Psi_2(t)|_{z=z_0}||_{L^p} \leq C||\nabla\Psi_0||_{H^3(\mathbb{R}_+^3)}  $.
\end{enumerate}
\end{prop}
\begin{proof}
\begin{enumerate}
\item We multiply $(QG_\nabla)$ by $\nabla\Psi$ and integrate.  By the properties of the projection operator $\mathbb{P}_\nabla$ and the divergence-free and stratified nature of the flow, $$ \int_{\mathbb{R}_+^3} \mathbb{P}_\nabla(\grad^\perp \Psi \cdot \grad(\nabla \Psi)) \cdot \nabla \Psi = \int_{\mathbb{R}_+^3} \grad^\perp \Psi \cdot \grad(\nabla \Psi) \cdot \nabla \Psi =  0 .$$
Therefore we obtain
\begin{align*}
\frac{1}{2}\frac{\partial}{\partial t} \int_{\mathbb{R}_+^3} |\nabla \Psi|^2 = \int_{\mathbb{R}_+^3} \nabla \Psi \nabla F &= -\int_{\mathbb{R}_+^3} \Psi \Delta F + \int_{\mathbb{R}^2} \Psi|_{z=0} \partial_\nu F \\
&= \int_{\mathbb{R}^2} \Psi|_{z=0} \lap \Psi|_{z=0} \\
&= - \int_{\mathbb{R}^2} |\grad \Psi|_{z=0}|^2
\end{align*}
Integrating in time then gives the claim.  
\item The estimate follows immediately from the transport equation for $\Delta\Psi$, the divergence free property of the flow, and Sobolev embedding.
\item We define $\tilde{\Psi}_2(z,x) = \Psi_2(|z|,x)$. Note that $\Delta\tilde{\Psi}_2(z,x) = \Delta\Psi_2(|z|,x)$ and $\nabla\tilde{\Psi}_2(z,x)=-\nabla\Psi_2(|z|,x)$. Applying the Riesz transforms to $\Delta\tilde{\Psi}_2$ and using (2) and parts (2) and (3) of \cref{corollaries} shows that $\nabla^2\tilde{\Psi}_2\in L^2(\mathbb{R}^3)$, and therefore $\nabla^2 \Psi_2 \in L^2(\mathbb{R}_+^3)$.  Applying \cref{trace} shows that 
\begin{equation}\label{eq:nabla}
\nabla\Psi_2 |_{z=0} \in \Hdot^\frac{1}{2}(\mathbb{R}^2)
\end{equation}
and it follows immediately from the Fourier characterization of $\Hdot^\frac{1}{2}(\mathbb{R}^2)$ that $\partial_x\Psi_2, \partial_y\Psi_2 \in \Hdot^\frac{1}{2}(\mathbb{R}^2)$ implies $(\lap)^\frac{3}{4}\Psi_2 |_{z=0} \in L^2(\mathbb{R}^2)$. 
\item We use again that $\Delta\tilde{\Psi}_2(z,x) = \Delta\Psi_2(|z|,x)$ together with parts (2) and (4) of \cref{corollaries} to obtain $\nabla\tilde{\Psi}_2|_{z=z_0} \in \Bdot_{\infty,\infty}^1(\mathbb{R}_+^3)$.  Using part (3) of \cref{equivalences} with $s=1$, $n=3$, and $k=2$ and recalling that $\nabla\tilde{\Psi}_2(z,x)=\nabla\Psi_2(|z|,x)$, we have $\nabla\Psi_2|_{z=z_0}\in \Bdot_{\infty,\infty}^1(\mathbb{R}^2)$.  
\item To obtain (5), we can use (3) and (4).  Using (4), \cref{equivalences}, and the Riesz transform shows that $-(\lap)^\frac{3}{4}\Psi_2 \in \Cdot^\frac{1}{2}(\mathbb{R}^2)$.  To show that $-(\lap)^\frac{3}{4}\Psi_2$ actually belongs to the inhomogenous space $C^\frac{1}{2}$, we must show that $-(\lap)^\frac{3}{4}\Psi_2\in L^\infty(\mathbb{R}^2)$.  This follows from (3) and the $\Cdot^\frac{1}{2}$ bound.  
\item We take the equation on the boundary $z=0$, multiply by $\dnu(t)$, and apply (3), yielding
\begin{align*}
\frac{1}{2}\frac{\partial}{\partial t} \int_{\mathbb{R}^2} |\dnu(t)|^2 &= \int_{\mathbb{R}^2} \lap \Psi(t) \dnu(t) \\
&= \int_{\mathbb{R}^2} \lap \Psi_1(t) \dnu_1(t) + \int_{\mathbb{R}^2} \lap \Psi_2(t) \dnu_1(t)\\
&= - \int_{\mathbb{R}^2} |\half \dnu_1(t)|^2 + \int_{\mathbb{R}^2} \lap\Psi_2(t)\dnu_1(t)  \\
&\leq -||\dnu_1(t)||_{\Hdot^\frac{1}{2}(\mathbb{R}^2)}^2 + ||\lap \Psi_2(t) ||_{\Hdot^{-\frac{1}{2}}(\mathbb{R}^2)} ||\dnu_1(t)||_{\Hdot^\frac{1}{2}(\mathbb{R}^2)} \\
&\leq -||\dnu_1(t)||_{\Hdot^\frac{1}{2}(\mathbb{R}^2)}^2 + ||(-\lap)^\frac{3}{4} \Psi_2(t) ||_{L^2(\mathbb{R}^2)}^2 + ||\dnu_1(t)||_{\Hdot^\frac{1}{2}(\mathbb{R}^2)}^2 \\
&\leq C||\nabla\Psi_0||_{H^3(\mathbb{R}_+^3)}
\end{align*}
Integrating in time finishes the proof.  
\item The estimate follows from \cref{eq:nabla}, Sobolev embedding, (4), and interpolation.  
\end{enumerate}
\end{proof}

Finally, let us remark that constants $C$ may change from line to line; if we wish to keep track of dependencies, we will write $C(\cdot)$.

\section{H\"{o}lder Regularity on the Boundary} 

Let us examine $\dnu_1=\dnu$.  We have that $\dnu_1$ satisfies the equation
$$\partial_t(\dnu_1) + \grad^\perp \Psi|_{z=0} \cdot \grad (\dnu_1) + \half (\dnu_1) =  \lap \Psi_2|_{z=0}.$$
Recalling that applying a continuation principle will require uniform in time regularity bounds on $\nabla\Psi$, we begin by showing the following regularity estimate on $\dnu$.
\begin{lemma}[H\"{o}lder Estimate]\label{holder}
If $\nabla\Psi \in L^\infty([0,T];H^3(\mathbb{R}_+^3))$ solves (QG) on $[0,T]$, there exists $r>0$, $C>0$ depending only on $||\nabla\Psi_0||_{H^3(\mathbb{R}_+^3)}$ such that the following holds.   The solution $\dnu$ to the boundary equation
$$\partial_t (\dnu) + \grad^\perp \Psi \cdot \grad (\dnu) =  \lap \Psi$$ satisfies $ \dnu \in C^r([0,T]\times \mathbb{R}^2)$ with $||\dnu||_{C^r([0,T]\times\mathbb{R}^2)} < C$.  
\end{lemma}
The steps of the De Giorgi argument are written for equations of the type
$$\partial_t \theta + u \cdot \grad \theta + \half \theta = f.$$
We will apply the De Giorgi lemmas to $\theta=\dnu$ to obtain \cref{holder}.  Estimates for $\theta$, $u$, and $f$ will come from \cref{apriori}; in particular, they will only depend on $||\nabla\Psi_0||_{H^3(\mathbb{R}_+^3)}$. We begin with the first De Giorgi lemma which will give an estimate on $||\theta||_{L^\infty([0,T]\times\mathbb{R}^2)}$.  Let us remark that all parts of the De Giorgi argument will be applied on the interval for which (QG) has a solution $\nabla \Psi \in L^\infty([0,T];H^3(\mathbb{R}_+^3))$.  From the trace, we have $\dnu \in L^\infty([0,T];H^\frac{5}{2}(\mathbb{R}_+^3))$, which justifies the calculations.  We begin with a technical proposition which we shall use several times to estimate the forcing term.

\begin{prop}\label{RHS}
Suppose that $||-(\bar{\Delta})^{-\frac{1}{4}}g(t,x)||_{L^\infty([-2,0];C^\frac{1}{2}(\mathbb{R}^2))} \leq M$, and $\omega(t,x)$ satisfies the following:
\begin{enumerate}
\item $\omega(t,x)\in L^\infty([-2,0];L^2\cap L^1(\mathbb{R}^2))\cap L^2([-2,0];\Hdot^\frac{1}{2}(\mathbb{R}^2))$
\item $\omega(t,x) \geq 0$ for $(t,x) \in [-2,0]\times\mathbb{R}^2$
\item  For each time $t$, $|\{(t,x): \omega(t,x)> 0\}|<\infty$
\end{enumerate}
Then 
$$ \int_{\mathbb{R}^2} g(t,x) \omega(t,x) \leq C(M) \left(\int_{\mathbb{R}^2}\omega(t,x)\,dx + \int_{\mathbb{R}^2} \mathcal{X}_{\{\omega(t,x)>0\}}\,dx \right) + \frac{1}{2}||\omega(t,\cdot)||_{\mathring{H}^\frac{1}{2}}^2 $$

\end{prop}

\begin{proof}
Put $h(t,x)= (-\lap)^{-\frac{1}{4}} g(t,x)$; we then have
\begin{align*}
I = \int_{\mathbb{R}^2}{g(t,x) \omega(t,x)\,dx} &= \int_{\mathbb{R}^2}{\eighth h(t,x) \eighth\omega(t,x)\,dx}\\
&\leq \int_{\mathbb{R}^2}\int_{\mathbb{R}^2}\frac{|h(t,x)-h(t,y)||\omega(t,x)-\omega(t,y)|}{|x-y|^\frac{5}{2}}\,dx\,dy \\
&= \iint_{|x-y|\leq1}\frac{|h(t,x)-h(t,y)||\omega(t,x)-\omega(t,y)|}{|x-y|^\frac{5}{2}}\,dx\,dy \\
&+ \iint_{|x-y|>1}\frac{|h(t,x)-h(t,y)||\omega(t,x)-\omega(t,y)|}{|x-y|^\frac{5}{2}}\,dx\,dy \\
&= I_1 + I_2.
\end{align*}
We begin by estimating $I_2$.  Since $h(t,x)\in L^\infty(C^\frac{1}{2})$, we have that
\begin{align*}
I_2 \leq& \iint_{|x-y|>1}\frac{ 2M |\omega(t,x)-\omega(t,y)|}{|x-y|^\frac{5}{2}}\,dx\,dy \\
\leq& 4M \iint_{|x-y|>1}\frac{\omega(t,x)}{|x-y|^\frac{5}{2}}\,dx\,dy\\
\leq& 4M \int_1^\infty r^{-\frac{3}{2}}\,dr \int_{\mathbb{R}^2} \omega(t,x)\,dx\\
\leq& 4M \int_{\mathbb{R}^2}\omega(t,x)\,dx\\
\end{align*}
We must now estimate $I_1$.  Using the symmetry in $x$ and $y$ and the fact that $$|\omega(t,x)-\omega(t,y)|\leq(\mathcal{X}_{\{ \omega(t,x) >0 \}}+\mathcal{X}_{\{ \omega(t,y) >0 \}})|\omega(t,x)-\omega(t,y)|$$
we have 
$$ I_1 \leq 2 \iint_{|x-y|\leq 1}{\mathcal{X}_{\{\omega(x)>0\}} \frac{|h(t,x)-h(t,y)||\omega(t,x)-\omega(t,y)|}{|x-y|^\frac{5}{2}}}\,dx\,dy.$$
By Cauchy's inequality, we have 
\begin{align*}
I_1 \leq 4 &\iint_{|x-y|\leq 1}{\mathcal{X}_{\{\omega(x)>0\}}\frac{|h(t,x)-h(t,y)|^2}{|x-y|^2}}\,dx\,dy\\
+& \frac{1}{4} \iint_{|x-y|\leq 1}{\frac{|\omega(t,x)-\omega(t,y)|^2}{|x-y|^3}}\,dx\,dy.
\end{align*}
Using the $L^\infty(C^\frac{1}{2})$ regularity of $h$, we have that 
\begin{align*}
4 \iint_{|x-y|\leq 1}\mathcal{X}_{\{\omega(t,x)>0\}}&\frac{|h(t,x)-h(t,y)|^2}{|x-y|^2}\,dx\,dy\\
\leq& 4 \iint_{|x-y|\leq 1}{\mathcal{X}_{\{\omega(t,x)>0\}}\frac{M^2|x-y|}{|x-y|^2}}\,dx\,dy\\
\leq& 4M^2 \int_0^1 \,dr \int_{\mathbb{R}^2} \mathcal{X}_{\{\omega(t,x)>0\}}\,dx\\
\leq& 4M^2 \int_{\mathbb{R}^2} \mathcal{X}_{\{\omega(t,x)>0\}}\,dx
\end{align*}
and
$$ \frac{1}{4}\iint_{|x-y|\leq 1}{\frac{|\omega(t,x)-\omega(t,y)|^2}{|x-y|^3}}\,dx\,dy \leq \frac{1}{2}||\omega(t,\cdot)||_{\mathring{H}^\frac{1}{2}}^2,$$
concluding the proof.
\end{proof}

\begin{lemma}[Global $L^\infty$ bound]\label{rough}
For any $M>0$, there exists $L>0$ such that the following holds.  Let $\theta \in L^\infty([-2,0];H^\frac{5}{2}(\mathbb{R}^2))$ be a solution to 
$$\partial_t \theta + u \cdot \grad \theta + \half \theta =  f$$
with 
$$||\theta||_{L^\infty([-2,0];L^2(\mathbb{R}^2))} + ||(-\lap)^{-\frac{1}{4}} f||_{L^\infty([-2,0]; C^{\frac{1}{2}}(\mathbb{R}^2))}<M$$
and $\dive u=0$.  Then $\theta(t,x)\leq L$ for $(t,x)\in[-1,0]\times\mathbb{R}^2$.
\end{lemma}
\begin{proof}
The main tool in showing the $L^\infty$ bound is an energy inequality, which we now derive.  Fix a constant $c>0$, and define $\theta_c:=(\theta-c)_+$.  Multiplying the equation by $\theta_c$, integrating in space, and using that the drift is divergence-free, we obtain

 $$ \frac{1}{2}\frac{d}{dt}\int_{\mathbb{R}^2}\theta_c^2(t,x)\,dx + \int_{\mathbb{R}^2}{\theta_c(t,x)\half \theta(t,x)\,dx} = \int_{\mathbb{R}^2} {f(t,x)\theta_c(t,x)\,dx}. $$
Making use of a pointwise estimate of C\'{o}rdoba and C\'{o}rdoba \cite{cc}, we have that
\begin{equation}\label{eq:1}
\frac{1}{2}\frac{d}{dt}\int_{\mathbb{R}^2}\theta_c^2(t,x)\,dx + \int_{\mathbb{R}^2}{|\quarter \theta_c(t,x)|^2\,dx} \leq \int_{\mathbb{R}^2} {f(t,x)\theta_c(t,x)\,dx}.
\end{equation}

Applying \cref{RHS} with $g=f$, $\theta_c=\omega$, we obtain the energy inequality
\begin{equation}\label{eq:energy}
\frac{d}{dt}\int_{\mathbb{R}^2}\theta_c^2(t,x)\,dx + \int_{\mathbb{R}^2}{|\quarter \theta_c(t,x)|^2\,dx} \leq C(M) \left(\int_{\mathbb{R}^2}\theta_c(t,x)\,dx + \int_{\mathbb{R}^2} \mathcal{X}_{\{\theta_c(t,x)>0\}}\,dx \right).
\end{equation}

With the energy inequality in hand, we obtain the desired nonlinear recurrence relation on the superlevel sets of energy.  Let $L>1$ be specificed later, and put $L_k=L(1-2^{-k})$, $\theta_k=(\theta-L_k)_+$ and $T_k=-1-2^{-k}$.  Define 
$$ E_k= \sup_{t\in[T_k,0]} \int_{\mathbb{R}^2} \theta_k^2(t,x)\,dx + \int_{T_k}^{0}\int_{\mathbb{R}^2}|\quarter \theta_k(\tau,x)|^2\,dx\,d\tau.$$
Choose $s\in[T_{k-1},T_k]$ and $t\in[T_k,0]$.  Integrating \eqref{eq:energy} from $s$ to $t$ yields 
\begin{align*}
&\int_{\mathbb{R}^2} \theta_k^2(t,x)\,dx + \int_{s}^{t}\int_{\mathbb{R}^2}|\quarter \theta_k(\tau,x)|^2\,dx\,d\tau \\
&\hspace{.3in} \leq \int_{\mathbb{R}^2} \theta_k^2(s,x)\,dx + C(M) \left(\int_s^t \int_{\mathbb{R}^2}|\theta_k(\tau,x)|\,dx\,d\tau + \int_s^t \int_{\mathbb{R}^2} \mathcal{X}_{\{\theta_k(\tau,x)>0\}}\,dx\,d\tau \right).
\end{align*}
Now taking the supremum on the left hand side, discarding the energy at time $s$, and averaging over $s\in[T_{k-1},T_k]$ on the right hand side, we have 
\begin{equation}\label{eq:nonlinear}
E_k \leq C(M)2^k \left(\int_{T_{k-1}}^0 \int_{\mathbb{R}^2} \theta_k(\tau,x) \,dx\,d\tau + \int_{T_{k-1}}^0 \int_{\mathbb{R}^2} \mathcal{X}_{\{\theta_k(\tau,x)>0\}}\,dx\,d\tau \right).
\end{equation}

We must control the right-hand side of \eqref{eq:nonlinear} by $E_{k-1}$ in a nonlinear fashion.  First, note that Sobolev embedding gives that $H^\frac{1}{2}(\mathbb{R}^2) \subset L^4(\mathbb{R}^2)$, and using this estimate to interpolate, we obtain
$$ ||\theta_k||_{L^3([T_k,0]\times\mathbb{R}^2)} \leq  E_k^\frac{1}{2}. $$
Next, we have that if $\theta_k > 0$, then $\theta_{k-1} \geq 2^{-k}L$ and $\mathcal{X}_{\{ \theta_{k} >0 \}} \leq \frac{2^k}{L} \theta_{k-1}$.  Therefore
$$ \int_{T_{k-1}}^0 \int_{\mathbb{R}^2} \theta_k \,dx\,d\tau \leq \int_{T_{k-1}}^0 \int_{\mathbb{R}^2} \theta_k \mathcal{X}_{\{ \theta_{k}>0 \}}^2 \,dx\,d\tau \leq \frac{4^k}{L^2} \int_{T_{k-1}}^0 \int_{\mathbb{R}^2} \theta_{k-1}^3 \,dx\,d\tau \leq \frac{4^k}{L^2} E_{k-1}^\frac{3}{2} $$
and
$$ \int_{T_{k-1}}^0 \int_{\mathbb{R}^2} \mathcal{X}_{\{ \theta_{k}>0 \}} \,dx\,d\tau = \int_{T_{k-1}}^0 \int_{\mathbb{R}^2} \mathcal{X}_{\{ \theta_{k}>0 \}}^3 \,dx\,d\tau \leq \frac{8^k}{L^3} \int_{T_{k-1}}^0 \int_{\mathbb{R}^2} \theta_{k-1}^3 \,dx\,d\tau \leq \frac{8^k}{L^3} E_{k-1}^\frac{3}{2}. $$ 
Combining these estimates, we have 
$$ E_k \leq \frac{C(M)^k}{L^2} E_{k-1}^\frac{3}{2}. $$ Depending only on $M$, we can choose $L$ to be large enough  and use \eqref{eq:nonlinear} to show that $E_1$ can be made small enough such that $\lim_{k\rightarrow\infty}E_k=0$.  Thus $||\theta_k||_{L^3([T_k,0]\times\mathbb{R}^2)}$ converges to zero as $k\rightarrow \infty$, and applying the Lebesgue dominated convergence theorem shows that $\theta\leq L$ on $[-1,0]\times\mathbb{R}^2$, proving the claim.
\end{proof}

To accommodate the second De Giorgi lemma, we must reformulate the $L^\infty$ bound.  The nonlocality of the equation makes the zooming arguments more delicate;  since the decrease in oscillation required for H\"{o}lder regularity will be nonlocal in nature, we cannot use a sharp cutoff as in \cref{rough}.  To address this, we will make use of a suitable cutoff function. Let $c(x)\in C^\infty(\mathbb{R}^2)$ such that $c\geq 0$, $c=0$ on $B_\frac{7}{4}(0)$, $c(x)=(|x|^\frac{1}{4}-2)_+$ for $|x|\geq 3$, and $c(x)\geq (|x|^\frac{1}{4}-2)_+$ for $|x|\geq 2$. We claim that $||\half c||_{L^\infty} < \infty$.  Using that $\nabla c \in C^\alpha(\mathbb{R}^2)$ for any $\alpha<1$, applying the Riesz transform, and using part (1) of \cref{corollaries} shows that $\half c \in BMO \cap \Cdot^\alpha(\mathbb{R}^2) \subset L^\infty(\mathbb{R}^2)$.  This cutoff function introduces an additional difficulty in that the drift term does not disappear after multiplying the equation by $\theta - c(x)$ and integrating.  Since $\dnu_1$ is now $L^\infty$ and $\grad^\perp\Psi_2 \in L^\infty$ by \cref{apriori}, the Riesz transform gives that $\grad^\perp \Psi \in \BMO$.  Performing a change of variables which follows the mean value of the drift through time, the new drift term will be exponentially integrable. Since $\BMO$ bounds are invariant under rescalings as well, following the flow at each successive dilation provides the needed uniform estimates.  With this in mind,  we can obtain the following sharper $L^\infty$ bound.  

\begin{lemma}[Local $L^\infty$ bound]\label{fine}
For any $C^*>0$, there exists $\delta>0$ such that the following holds. Let $\theta \in L^\infty([-2,0];H^\frac{5}{2}(\mathbb{R}^2))$ be a solution to 
$$\partial_t \theta + u \cdot \grad \theta + \half \theta =  f$$
such that $\theta(t,x) \leq 1+c(x)$ on $[-2,0]\times \mathbb{R}^2$, $\dive u=0$ and
$$||(-\lap)^{-\frac{1}{4}} f||_{L^\infty([-2,0];C^\frac{1}{2})(\mathbb{R}^2)} + ||u||_{L^\infty([-2,0];L^4(B_2(0)))} \leq C^*.$$ 
If 
$$ |\{{\theta}>0\}\cap ([-1,0]\times B_1(0))| < \delta,$$
then $\theta(t,x)\leq \frac{1}{2}$ for $(t,x)\in[-\frac{1}{2},0]\times B_\frac{1}{2}(0)$.
\end{lemma}

\begin{proof}
Let $\gamma_k$ be a bump function compactly supported in $B_{\frac{1}{2}+2^{-k-1}}$, equal to $\frac{1}{2}+2^{-k-1}$ on $B_{\frac{1}{2}+2^{-k-2}}$, with $0\leq \gamma_k \leq \frac{1}{2}+2^{-k-1}$ for all $x$, and $\gamma_k<\gamma_l$ for $k>l$.  We will also impose that $|\grad \gamma_k|\leq C2^k$ and $|\half \gamma_k| \leq Ck2^{k}$ (we provide a short justification of this condition in the appendix after the discussion of \cref{inequality}).  Define ${\theta}_k:= ({\theta}-(1+c-\gamma_k))_+$.  We multiply the equation by ${\theta}_k$ and argue as before.  First we record the following estimates: 

\begin{align}
\int_{\mathbb{R}^2} \half {\theta} {\theta}_k \,dx \leq& \int_{\mathbb{R}^2} |\quarter {\theta}_k|^2\,dx + \int_{\mathbb{R}^2} \half (1+c-\gamma_k){\theta}_k\,dx \nonumber \\
       \leq& \hspace{.1in} \int_{\mathbb{R}^2} |\quarter{\theta}_k|^2 \,dx + \left(||\half c||_{L^\infty}+Ck2^{k}\right) \int_{\mathbb{R}^2} {\theta}_k\,dx \label{eq1}
\end{align}
and
\begin{align}
\int_{\mathbb{R}^2}{\left({u}\cdot \grad {\theta}\right){\theta}_k}\,dx =& \int_{\mathbb{R}^2}{\left({u}\cdot \grad {\theta}_k\right) {\theta}_k}\,dx \nonumber +\int_{\mathbb{R}^2}{\left( {u}\cdot \grad (1+c-\gamma_k)\right) {\theta}_k}\,dx \nonumber \\
=&\int_{\mathbb{R}^2}{\left( {u}\cdot \grad (1+c-\gamma_k)\right) {\theta}_k}\,dx \nonumber \\
\leq& \hspace{.1in} C(C^*)\left(||\grad c||_{L^\infty}+||\grad \gamma_k||_{L^\infty}\right) \left(\int_{\mathbb{R}^2} {\theta}_k^\frac{4}{3} \,dx\right)^\frac{3}{4}. \label{eq2}
\end{align}
after applying H\"{o}lder's inequality with $q=\frac{4}{3}$ on $\theta$ and $\frac{q-1}{q}=4$ on $u$. In addition, we can estimate the term
\begin{equation}\label{forcing}
\int_{\mathbb{R}^2} f {\theta}_k \,dx 
\end{equation}
using \cref{RHS}.  Combining estimates for \eqref{eq1}, \eqref{eq2}, and \eqref{forcing}, we arrive at the energy inequality
\begin{align}
\frac{d}{dt}\int_{\mathbb{R}^2}{\theta}_k^2\,dx &+ \int_{\mathbb{R}^2}{|\quarter {\theta}_k|^2\,dx} \leq \nonumber\\
&C(C^*)\left(2^k+k2^{k}\right) \left(\int_{\mathbb{R}^2}{\theta_k}\,dx + \int_{\mathbb{R}^2} \mathcal{X}_{\{{\theta}_k >0\}}\,dx + \big( \int_{\mathbb{R}^2} {\theta}_k^q \,dx \big)^\frac{1}{q} \right). \label{fineenergy}
\end{align}

Now define $T_k:= -\frac{1}{2}-2^{-k-1}$, and put
$$ E_k:= \sup_{t\in[T_k,0]} \int_{\mathbb{R}^2} {\theta}_k^2(t,x)\,dx + \int_{T_k}^{0}\int_{\mathbb{R}^2}|\quarter {\theta}_k(\tau,x)|^2\,dx\,d\tau.$$
Integrating in time, we have that the first two terms on the right hand side can be estimated as in \cref{rough}.  Using Jensen's inequality and the fact that $\mathcal{X}_{\{ \theta_k>0 \}} \leq 2^{k}\theta_{k-1}$, we can estimate the third term on the right hand side of \eqref{fineenergy} by
\begin{align*}
\int_{T_{k-1}}^0 \left( \int_{\mathbb{R}^2} {\theta}_k^q \,dx \right)^\frac{1}{q} \,d\tau \leq& \int_{T_{k-1}}^0 \left( \int_{\mathbb{R}^2} {\theta}_k^q \mathcal{X}_{\{ {\theta}_{k}>0 \}}^{3-q} \,dx \right)^\frac{1}{q} \,d\tau \\
\leq& \hspace{.05in}C \left( \int_{T_{k-1}}^0 \int_{\mathbb{R}^2} {\theta}_k^q \mathcal{X}_{\{ {\theta}_{k}>0 \}}^{3-q} \,dx\,d\tau \right)^\frac{1}{q} \\
\leq& \hspace{.05in}C \left( \left(2^{k} \right)^{3-q} \int_{T_{k-1}}^0 \int_{\mathbb{R}^2} {\theta}_{k-1}^3 \,dx\,d\tau \right)^\frac{1}{q}\\
\leq& \hspace{.05in} C^k E_{k-1}^\frac{3}{2q}.
\end{align*}
Using the integrability assumption on $u$, and recalling that $q=\frac{4}{3}$, the nonlinear recurrence relation on $E_k$ follows as in \cref{rough}.  Noticing that \eqref{fineenergy} shows that choosing $\delta$ arbitrarily small makes $E_0$ arbitrarily small, there exists $\delta$ such that $\lim_{k\rightarrow\infty}E_k =0$.  Therefore, $\theta_k$ converges to $0$ in $L^2$ for every time $t\in[-\frac{1}{2},0]$. Thus ${\theta} \leq \frac{1}{2}$ on $[-\frac{1}{2},0]\times B_\frac{1}{2}(0)$.
\end{proof}

With the $L^\infty$ bound in hand, we turn to the second half of the De Giorgi argument.  Let us remark the argument in \cite{ccv} works for kernels comparable to the fractional Laplacian $\fractional$  raised to any power $\alpha \in (0,1)$.  In addition, Schwab and Silvestre \cite{Schwab2016}) proved a regularity result for parabolic equations assuming that the drift and the forcing were bounded.  Without bounded drift and forcing, we cannot follow \cite{Schwab2016}. However, we have more dissipation than is necessary for the argument in \cite{ccv}.  Therefore, we can instead make a compactness argument following \cite{v}. Since $\alpha = \frac{1}{2}$, the solutions belong to $H^\frac{1}{2}$, and we can make use of \cref{bbm}. This will show that the energy cannot increase or decrease too rapidly in time.    

First, a parabolic version of the isoperimetric lemma will be shown, following the proof in \cite{v}.  This will then imply that $\theta$ enjoys a geometric rate of decrease in oscillation.  
Let $\phi$ be a compactly supported, radially symmetric and decreasing, $C^\infty$ bump function such that $0\leq\phi(x) \leq 1$ for all $x$, $\phi=1$ on $B_1(0)$, and $\supp\phi\subset B_\frac{3}{2}(0)$.  Let $\phi_0(x) = 1+c(x)-\phi(x)$, and $\phi_1(x)=1+c(x)-\frac{1}{2}\phi(x)$.

\begin{lemma}[Isoperimetric Lemma]\label{isoperimetric}
For any $C^*,\beta>0$ there exists $\alpha$ such that the following holds.  Let ${\theta}\in L^\infty([-2,0];H^\frac{5}{2}(\mathbb{R}^2))$ solve
$$\partial_t \theta + u \cdot \grad \theta + \half \theta =  f$$
with ${\theta}(x) \leq 1+c(x)$. Assume that  
$$||(-\lap)^{-\frac{1}{4}} f||_{L^\infty([-2,0];C^\frac{1}{2})(\mathbb{R}^2)} + ||u||_{L^\infty([-2,0];L^4(B_2(0)))} \leq C^*$$ 
and $\dive u=0$.  Fix $\delta$ as in \cref{fine}.   Define
\begin{align*}
A=& \{{\theta} > \frac{1}{2} \}\cap ([-1,0]\times B_1)\\
C=& \{{\theta} \leq 0 \} \cap ([-2,-1]\times B_1 )\\
D=& \{ \phi_0 < \theta \leq \phi_1 \} \cap ( [-2,0]\times B_2 ).
\end{align*}
Then if $|A|\geq \delta$, $|C|\geq \beta$, then $|D|\geq \alpha$.
\end{lemma}
\begin{proof}
Assume that the lemma is false.  Then, given $\beta$ there exists a sequence of solutions ${\theta}_j$ such that $|A_j| \geq\delta$, $|C_j| \geq\beta$, $|D_j|  \leq \frac{1}{j}$ with $A_j$, $C_j$, and $D_j$ defined analogously to $A$, $C$, and $D$.  Put $v_j = ({\theta}_j - \phi_0)_+$.  The proof will use the Aubin-Lions compactness lemma \cite{a} to extract a subsequential limit which will satisfy the energy inequality but does not take values in between $0$ and $\frac{1}{2}$, reaching a contradiction.  

In order to apply the Aubin-Lions lemma to $v_j^2$, we show that ${\partial_t}v_j^2 \in L^1([-2,0];H^{-2}(B_2(0)))$, $v_j^2 \in L^2(H^\frac{1}{2}(B_2(0)))$, and $v_j^2 \in L^\infty([-2,0];L^2(B_2(0)))$.  The third criterion is immediate from the assumptions, so we focus on the first and second.   We multiply by $v_j$ and integrate.  The $L^\infty(L^4)$ bound on ${u}$ gives that 
\begin{align*}
\int_{\mathbb{R}^2} ({u}\cdot \grad {\theta}_j) v_j \,dx =& \int_{\mathbb{R}^2} ({u}\cdot \grad v_j) v_j \,dx +  \int_{\mathbb{R}^2} \left({u}\cdot \grad(1+c-\phi)\right) v_j \,dx \\
=& \int_{\mathbb{R}^2} \left({u}\cdot \grad(1+c-\phi)\right) v_j \,dx \\
\leq& \hspace{.05in} C(C^*, \phi, c)
\end{align*}
after using the compact support of $v_j$, the bounds on $c$ and $\phi$, and H\"{o}lder's inequality.  
We can estimate the forcing term using \cref{RHS} by setting $f=g$ and $v_j=\omega$.  Using the $L^\infty$ bound on $v_j$ and absorbing the $\Hdot^\frac{1}{2}$ norm into the left hand side, we obtain the energy inequality
$$ \frac{d}{dt} \int_{\mathbb{R}^2}{v_j^2} + \int_{\mathbb{R}^2}{|\quarter v_j|^2} \leq  C(C^*, \phi, c).$$  
Integrating from $s$ to $t$ in time for $-2<s<t<0$ gives
\begin{equation}\label{compactness}
\int_{\mathbb{R}^2}{v_j^2(t)} + \int_s^t \int_{\mathbb{R}^2}{|\quarter v_j|^2} \leq \int_{\mathbb{R}^2}{v_j^2(s)} + C(C^*, \phi, c)(t-s).
\end{equation}
This implies that $v_j$ is uniformly bounded in  $L^2([-2,0];\Hdot^\frac{1}{2}(\mathbb{R}^2))$.  Also, note that since $0 \leq v_j \leq 1$, for all $x,y$ we have
$$ |v_j^2(x)-v_j^2(y)|^2 \leq  4|v_j(x)-v_j(y)|^2$$
Examining the Gagliardo seminorm shows then that $|| v_j^2 ||_{\Hdot^\frac{1}{2}(\mathbb{R}^2)} \leq 4|| v_j ||_{\Hdot^\frac{1}{2}(\mathbb{R}^2)} $ (see \cite{hitchhiker} for details concerning equivalent definitions of fractional Sobolev spaces).  By restriction, we have that $|| v_j^2 ||_{\Hdot^\frac{1}{2}(B_2(0))} \leq 4|| v_j ||_{\Hdot^\frac{1}{2}(B_2(0))} $, and so ${v_j^2}$ is uniformly bounded in $L^2([-2,0];H^\frac{1}{2}(B_2(0)))$ by a constant depending only on $C^*, \phi, c$.  

We now show that $\partial_t {v_j^2} \in L^1([-2,0];H^{-2}(B^2(0)))$; here $H^{-2}(B_2(0))$ denotes the dual of $H^2_0(B_2(0))$.  Multiplying the equation by $v_j$, we obtain
$$ \frac{1}{2} \partial_t v_j^2 = - \dive(u \theta_j) v_j - \half \theta_j v_j + f v_j. $$
We must show that each term on the right hand side belongs to $L^1([-2,0];H^{-2}(B_2(0)))$.  
\begin{enumerate}
\item Note that $\dive(u \theta_j) v_j = \frac{1}{2}\dive (u v_j^2) + u \cdot \nabla \phi_0 v_j$.  Since $v_j^2 \in L^\infty([-2,0]\times\mathbb{R}^2)$ is compactly supported and $u\in L^\infty([-2,0];L^4(\mathbb{R}^2))$,  part (2) of \cref{corollaries} shows that $\dive (u v_j^2)$ belongs to $ L^1([-2,0];\Hdot^{-1}(\mathbb{R}^2) \subset L^1([-2,0];H^{-2}(B_2(0))) $.  Also, since $\nabla \psi_0 $ is smooth and compactly supported, it is immediate that $u \cdot \nabla \phi_0 v_j \in L^1([-2,0];L^2(\mathbb{R}^2))$.  Therefore $ \dive( u \theta_j)v_j \in L^1([-2,0];H^{-2}(B_2(0)))$.  
\item Since $(-\lap)^{-\frac{1}{4}}f\in L^\infty([-2,0];L^\infty(\mathbb{R}^2))$ and $v_j \in L^\infty([-2,0]\times\mathbb{R}^2) \cap L^\infty([-2,0];\Hdot^\frac{1}{2}(\mathbb{R}^2))$ is compactly supported in $B_2(0)$, we can apply part (1) of \cref{duality} with $z=v_j$ and $f=w$ to conclude that $f v_j$ is uniformly bounded in $ L^1([-2,0];H^{-2}(\mathbb{R}^2)) \subset L^1([-2,0];H^{-2}(B_2(0)))$.
\item We have that $$ \half \theta_j v_j = \half[(\theta_j-\phi_0) - (\theta_j-\phi_0)_+] v_j + \half v_j v_j + \half \phi_0 v_j  .$$  Since $\half \phi_0 \in L^\infty([-2,0];L^2(\mathbb{R}^2))$ and $v_j\in L^\infty([-2,0];L^2(\mathbb{R}^2))$, it is immediate that $\half \phi_0 v_j \in L^1([-2,0];L^2(B_2(0)))$. In addition, we can apply part (2) of \cref{duality} to the second term to obtain that $\half v_j v_j \in L^1([-2,0];H^{-2}(B_2(0))$. In order to estimate the first term, first note that the pointwise estimate of C\'{o}rdoba and C\'{o}rdoba \cite{cc} shows that $\half[(\theta_j-\phi_0) - (\theta_j-\phi_0)_+] v_j$ is a positive measure on $B_3(0)$ for each time $t\in[-2,0]$.  We first show that $\half[(\theta-\phi_0) - (\theta-\phi_0)_+] v_j \in L^1([-2,0];\mathcal{M}(B_3(0)))$; here $\mathcal{M}(B_3(0))$ is the Banach space of all Borel measures on $B_3(0)$ with the total variation norm.  We have that $$  \half[(\theta-\phi_0) - (\theta-\phi_0)_+] v_j = -\frac{1}{2}{\partial_t} v_j^2 - (u\cdot \grad \theta_j) v_j - \half v_j v_j - \half \phi_0 v_j + f v_j .$$ To show that $\half[(\theta-\phi_0) - (\theta-\phi_0)_+] v_j \in L^1([-2,0];\mathcal{M}(B_3(0)))$, we multiply by $\mathcal{X}_{B_3(0)}$ and integrate in space and time.  Note that since each term on the right hand side contains a factor of $v_j$ which is compactly supported in $B_2(0)$, multiplying by $\mathcal{X}_{B_3(0)}$ has no effect.  First, we have that 
$$ \int_{-2}^0 \int_{B_3(0)} \partial_t v_j^2 = \int_{B_3(0)} \int_{-2}^0  \partial_t v_j^2  = \int_{B_3(0)} v_j^2(0) - v_j^2(-2) \leq 2 |B_3(0)|.$$
Here we have used the a priori regularity assumptions, the $\Hdot^\frac{1}{2}$ bound on $v_j$, and the equality
$$ \frac{1}{2} \partial_t v_j^2 = - u \cdot\grad\theta_j v_j - \half \theta_j v_j + f v_j $$
to justify integrating $\partial_t v_j^2$ in space and time.  Next, splitting $u\cdot \grad \theta_j v_j = u \cdot \grad \phi_0 v_j + u \cdot \grad v_j v_j$ and integrating by parts shows that $$  \int_{-2}^0 \int_{B_3(0)} u \cdot \grad \theta_j v_j \leq ||u||_{L^\infty(L^4)}. $$
Since $\half \phi_0 v_j$ is bounded, multiplying by $\mathcal{X}_{B_3(0)}$ and integrating produces at most a fixed constant depending only on $\phi_0$.   Also, we have that $$ - \int_{-2}^0 \int_{B_3(0)} \half v_j v_j $$ is negative and may be discarded.  Finally, applying \cref{RHS} with $g=f$, $\omega = v_j $ and using the $L^2(\Hdot^\frac{1}{2})$ and $L^\infty$ bounds on $v_j$ shows that $$ \int_{-2}^0 \int_{B_3(0)} f v_j \leq C(C^*, \phi, c).  $$
Therefore, $\half[(\theta-\phi_0) - (\theta-\phi_0)_+] v_j$ is bounded in $L^1([-2,0];\mathcal{M}(B_3(0)))$ by a constant depending only on $C^*, \phi, c$.  By Sobolev embedding, $H^2_0(B_2(0)) \subset C_c (B_3(0))$.  Recalling then that $\mathcal{M}(B_3(0)) = C_c (B_3(0))^*$, we have that $L^1([-2,0];\mathcal{M}(B_3(0))) \subset L^1([-2,0];H^{-2}(B_2(0)))$, showing that $\half[(\theta-\phi_0) - (\theta-\phi_0)_+] v_j$, and therefore $\half \theta_j v_j$, belong to $ L^1([-2,0];H^{-2}(B_2(0)))$.  
\end{enumerate}

Since $H^\frac{1}{2}(B_2(0))$ embeds compactly into $L^1(B_2(0))$ (see again \cite{hitchhiker}) and $L^1(B_2(0))$ embeds continuously into $H^{-2}(B_2(0))$, by the Aubin-Lions compactness lemma from \cite{a}, up to a subsequence, $v_j^2$ converges in $L^2([-2,0];L^1(B_2(0)))$ to a function $v^2$. Passing to the limit in \eqref{compactness} shows that $v^2$ satisfies the inequality $$ \int_{B_2}{v^2(t)}  \leq \int_{B_2}{v^2(s)} + C(C^*, \phi, c)(t-s)$$
for $s<t$.  By assumption, $\theta_j$ satisfies $|D_j|  \leq \frac{1}{j}$.  Therefore, $v_j$ then satisfies by definition 
$$ |\{0<v_j \leq\frac{1}{2}\phi \}\cap([-2,0]\times B_2)|\leq \frac{1}{j},$$ and $v_j^2$ satisfies
$$ |\{0<v_j^2 \leq\frac{1}{4}\phi^2 \}\cap([-2,0]\times B_2)|\leq \frac{1}{j}.$$
Using Tchebyshev's inequality and passing to the limit, we have that 
$$ |\{0<v^2 < \frac{1}{4}\phi^2 \}\cap([-2,0]\times B_2)|=0.$$ 
Since $v^2$ belongs to $L^2([-2,0];H^\frac{1}{2}(\supp\phi))$,  applying \cref{bbm} to $v^2$ shows that for almost every time $t\in[-2,0]$, either $v^2=0$ or $v^2 \geq \frac{1}{4}\phi^2 $.  Using that $|C_j|\geq \beta$ for every $j$, there must exist a positive measure set of times in $[-2,-1]$ for which $v=0$.  The energy inequality shows that as soon as $v=0$ for some time $s$, $v=0$ on all of $[s,0]\times B_2$, and thus we conclude that $v=0$ on $[-1,0]\times B_2$.  However, we also have that $|A_j|\geq \delta$ for all $j$, which bounds the norms of $v_j^2$ in $L^2([-1,0];L^1(B_2(0)))$ from below uniformly in $j$, contradicting the convergence of $v_j^2$ to $v^2$ in $L^2([-2,0];L^1(B_2(0)))$.  
\end{proof}

We turn now to the oscillation lemma.  We will make use of the cutoff function $c_\epsilon(x)= (|x|^\epsilon-2^{4\epsilon})_+$.

\begin{lemma}[Decrease in Oscillation]\label{decinosc}
For any $C^*$, there exists $\zeta\in(0,1)$, $\epsilon\in(0,\frac{1}{4})$, and $\eta>0$ such that the following holds. Let ${\theta}\in L^\infty([-2,0];H^\frac{5}{2}(\mathbb{R}^2))$ solve
$$\partial_t \theta + u \cdot \grad \theta + \half \theta =  f$$
with ${\theta}(x) \leq 1+c_\epsilon(x)$. Assume that  
$$\eta^{-1}||(-\lap)^{-\frac{1}{4}} f||_{L^\infty([-2,0];C^\frac{1}{2}(\mathbb{R}^2))} + ||u||_{L^\infty([-2,0];L^4(B_2(0)))} \leq C^*$$ 
and $\dive u=0$.   Let $\delta$, $\beta$, $\alpha$ be as in \cref{isoperimetric} with  $|C|\geq\beta$.  Then $\theta \leq 1-\zeta$ on $[-\frac{1}{2},0]\times B_\frac{1}{2}$.  
\end{lemma}
\begin{proof}
Choose $K$ such that $K\alpha > |(-2,0)\times B_2|$,  and let $\eta=2^{-K}$.  Put ${\theta}_k=2^k({\theta}-(1-2^{-k}))$.  By scaling, ${\theta}_k$ solves the equation
$$ \partial_t {\theta}_k + {u} \cdot \grad {\theta}_k + \half {\theta}_k = 2^k f. $$
For $k\leq K$, 
$$||2^k (-\lap)^{-\frac{1}{4}} f||_{L^\infty(C^\frac{1}{2})} \leq C^*.$$
Choose $\epsilon<<\frac{1}{4}$ to be small enough such that 
$$ 2^K(|x|^\epsilon-2^{4\epsilon})_+ \leq (|x|^\frac{1}{4}-2)_+ \leq c(x)$$ 
for all $x$.  Note that since $k\leq K$ we have 
$$ {\theta}_k(x) \leq 1 + 2^K c_{\epsilon}(x) \leq 1 + c(x).$$
Fix $k\leq K$ now, and suppose that 
\begin{equation}\label{contradiction}
| \{ {\theta}_{j+1} > 0 \}\cap ([-1,0]\times B_1) | \geq \delta.
\end{equation}
  for all $j\leq k$.  This implies that 
$$ |\{{\theta}_j > \frac{1}{2} \}\cap ([-1,0]\times B_1)| \geq \delta .$$
Since $|\{{\theta}_j \leq 0 \} \cap \{[-2,-1]\times B_1 \}|\geq \beta $ for all $j$, we have that by \cref{isoperimetric}, 
$$ |\{ \phi_0< \theta_j \leq \phi_1 \}\cap([-2,0]\times B_2)|\geq \alpha .$$
Noticing that the sets $\{ \phi_0 < \theta_j \leq \phi_1 \}$, $\{ \phi_0 < \theta_{j'} \leq \phi_1 \}$ are disjoint for $j \neq j'$, we have that \eqref{contradiction} cannot hold for $k=K$ by choice of $K$.  So there must exist $k<K$ for which 
$$| \{ {\theta}_{k+1} > 0 \}\cap ([-1,0]\times B_1) | < \delta .$$
 Then by \cref{fine}, $\theta_{k+1}\leq\frac{1}{2}$ in $[-\frac{1}{2},0]\times B_\frac{1}{2}$, and $\theta \leq 1-2^{-(2+K)}$ in $[-\frac{1}{2},0]\times B_\frac{1}{2}$, proving the claim with $\zeta = 2^{-(2+K)}$.  
\end{proof}

We have arrived at \cref{decinosccor} as an easy corollary.
  
\begin{lemma}\label{decinosccor}
If $-1-c_\epsilon \leq {\theta} \leq 1+c_\epsilon$ on $[-2,0]\times \mathbb{R}^2$ and the conditions of \cref{decinosc} are satisfied, then on $[-\frac{1}{2},0]\times B_\frac{1}{2}$, 
$$\sup \theta - \inf \theta \leq 2-\zeta .$$
\end{lemma}

We can now prove the main regularity estimate for $\dnu$. 

\begin{proof}[Proof of \cref{holder}]
Throughout the argument, $\theta$, $u$, and $f$ correspond to (dilated versions of) $\dnu$, $\grad^\perp\Psi$, and $\lap\Psi_2$, respectively. The regularity assumptions on each of the function in the De Giorgi lemmas is provided by \cref{apriori}; we give the details in the proof as they appear.  Recall that the local existence theorem guarantees the existence of some time $\bar{T}$ such that $(QG)$ admits a smooth solution on $[0,\bar{T}]$.  We assume that $\nabla\Psi\in L^\infty([0,T];H^\frac{5}{2}(\mathbb{R}^2))$ is a solution to $(QG)$, and $T\geq\bar{T}$.  We then choose $(t_0,x_0)\in [0,T]\times\mathbb{R}^2$ such that $t_0\geq\frac{\bar{T}}{2}$.   We define $K_0= \inf(1,\frac{t_0}{4})$.  Put $\theta_0(t,x)= \dnu(t_0+K_0t, x_0+K_0x) $, $u_0(t,x)= \nabla^\perp\Psi(t_0+K_0t, x_0+K_0x) $, and $f_0(t,x)=\lap \Psi_{2}(t_0+K_0t, x_0+K_0x)$.  Then by the a priori estimates in (5) and (6) of \cref{apriori}, $\theta_0$ (and $-\theta_0$) satisfy the assumptions of \cref{rough}, and we have that $\theta_0 \in L^\infty([-1,0]\times\mathbb{R}^2)$.  Since the argument is translation invariant in space and we can only need to consider times $t\in[\frac{\bar{T}}{2},T]$, we have in fact that $\theta \in L^\infty([0,T]\times\mathbb{R}^2)$, with $||\theta||_{L^\infty([0,T]\times\mathbb{R}^2)}$ depending only on $||\nabla\Psi_0||_{H^3(\mathbb{R}_+^3)}$.

Continuing to fix $(t_0, x_0)$ and $K_0$ as above, we will show that $\theta$ is H\"{o}lder continuous at $(t_0, x_0)$.  We will inductively define a sequence of dilated functions for some factor of dilation $K$ to be determined later.  Let $\Gamma_{1}(t)$ be the solution to the ODE
$$  \left\{
       \begin{array}{@{}l@{\thinspace}l}
       \dot{\Gamma}_{1}(t) = \dashint_{B_1(\Gamma_{1}(t))} u_0(Kt,Ky)\,dy\\
       \Gamma_{1}(0) = 0\
       \end{array}  \right.  $$
and put
$$ \theta_{1}(t,x) = \frac{\theta_0(Kt, Kx+\Gamma_1(t))}{||\theta||_{L^\infty}} $$
$$ u_1(t,x) = u_0(Kt, Kx+\Gamma_1(t)) $$
$$ \lap \Psi_{2,1}(t,x) = \frac{1}{||\theta||_{L^\infty}}f_0(Kt, Kx+\Gamma_1(t)) .$$
For $k>1$, define
$$  \left\{
       \begin{array}{@{}l@{\thinspace}l}
       \dot{\Gamma}_{k+1}(t) = \dashint_{B_1(\Gamma_{k+1}(t))} u_k(Kt,Ky)\,dy-\dashint_{B_1(0)}u_k(t)\\
       \Gamma_{k+1}(0) = 0\
       \end{array}  \right.  $$

 $$\theta_{k+1}(t,x) = \frac{1}{1-\frac{\zeta}{4}}\left( \theta_k\left(Kt, Kx+\Gamma_{k+1}(t)\right)- \frac{1}{2}\left( \sup_{[-\frac{1}{4},0]\times {B_\frac{1}{4}}} \theta_k + \inf_{[-\frac{1}{4},0]\times {B_\frac{1}{4}}} \theta_k \right)\right)$$ 
$$ u_{k+1}(t,x) = u_k\left(Kt, Kx + \Gamma_{k+1}(t)\right) $$  
$$ \lap \Psi_{2,k+1}(t,x) = \lap \Psi_{2,k}(Kt, Kx+\Gamma_{k+1}(t)) .$$
We have that $\theta_k$ solves the equation
$$ \partial_t \theta_k + (u_k-\dashint_{B_1(0)}u_k )\cdot \grad \theta_k + \half \theta_k = \left(\frac{K}{(1-\frac{\zeta}{4})}\right)^k \lap \Psi_{2,k} .$$
Examining the assumptions of the De Giorgi lemmas (\cref{fine}, \cref{isoperimetric}, \cref{decinosc}, and \cref{decinosccor}), we see that $K$ is subject to the following constraints.  

\begin{enumerate}
\item $K$ needs to be small enough to satisfy  
$$ \frac{1}{1-\frac{\zeta}{2}} c_\epsilon(Kx) < c_\epsilon (x) $$
for all $x\geq \frac{1}{K}$.  Recalling that $c_\epsilon(x)= (|x|^\epsilon-2^{4\epsilon})_+$, choosing $K\leq {1-\frac{\zeta}{2}}$ satisfies this constraint.
\item $K$ should be small enough so that $ \left(\frac{K}{(1-\frac{\zeta}{4})}\right)^k \lap \Psi_{2,k}$ satisfies the assumptions of the De Giorgi lemmas uniformly in $k$.  Specifically, we need $(-\lap)^{-\frac{1}{4}}\left(\left(\frac{K}{(1-\frac{\zeta}{4})}\right)^k \lap \Psi_{2,k}\right) \in L^\infty([-2,0];C^\frac{1}{2}(\mathbb{R}^2))$ to have small norm.  Applying $(-\lap)^{-\frac{1}{4}}$ divides by a factor of $K^\frac{k}{2}$, but we can choose $K$ to be very small compared to $(1-\frac{\zeta}{4})$, and by (3) and (5) of \cref{apriori}, we can choose $K$ to satisfy this constraint. 
\item We must ensure that $u_k-\dashint_{B_1(0)}u_k$ satisfies the assumptions of the De Giorgi lemmas uniformly in $k$.  Specifically, we must have that $u_k - \dashint_{B_1(0)} u_k \in L^\infty([-2,0];L^4(B_2(0)))$ uniformly in $k$.  Using that $\grad^\perp \Psi_1$ is related to $\theta$ by the Riesz transform, the $L^\infty$ bound on $\theta$, part (7) of \cref{apriori}, parts (1) and (2) of \cref{bmofacts}, and the scale invariance of the $\BMO$ norm, this condition is satisfied independent of $K$.
\item Notice that each successive dilation includes a change of variables which follows the new flow of the dilated drift term.  At the $k^{th}$ iteration, we will obtain a decrease in oscillation for $\theta_k$ on the set $[-\frac{1}{2},0]\times B_\frac{1}{2}(0)$.  Then after dilating by $K$ and shifting according to $\Gamma_{k+1}$, we must ensure that $-1-c_\epsilon \leq \theta_{k+1} \leq 1+c_\epsilon$.  Applying \cref{bmofacts}, we have that $|\dot{\Gamma}_{k+1}| < C$ for some fixed constant $C$.  Therefore we can choose $K$ small enough so that zooming in by a factor of $K$ and then shifting according to the new drift gives that $-1-c_\epsilon \leq \theta_{k+1} \leq 1+c_\epsilon$.  
\end{enumerate}

We choose $K$ to satisfy the above constraints.  Thus we have that $\{\theta_k\}_{k=1}^\infty$ satisfies the assumptions of the De Giorgi lemmas uniformly in $k$, and we obtain a decrease in oscillation of $1-\frac{\zeta}{4}$ for each successive iteration.  To see that $\theta$ is H\"{o}lder continuous, put
$$U_k= \sup_{[-2,0]}{\hspace{.05in} |\dot{\Gamma}_k(t)|}$$ 
and notice that the set $ [-\frac{1}{2},0]\times B_{\frac{1}{2}}(\Gamma_k(t)) $ contains the rectangle $[-\frac{1}{4U_k},0]\times B_{\frac{1}{4}}(0)$.  By \cref{bmofacts}, there exists $U$ such that $U_k \leq U$ for all $k$. Putting $D=\min{(\frac{K}{4}, \frac{1}{8U})}$, we have that if $(t,x)$ is such that
 $$|(t,x)-(t_0,x_0)|\approx D^{k}$$
 then
  $$|\theta(t,x)-\theta(t_0,x_0)| \leq  \left(1-\frac{\zeta}{4}\right)^k .$$
Therefore we have that $\theta$ is H\"{o}lder continuous at $(t_0,x_0)$ with exponent 
 $$ r = \frac{\log(1-\frac{\zeta}{4})}{\log(D)} .$$
We have that $r$ depends only on the parameters $M$ and $C^*$, which in turn depend only on $||\nabla\Psi_0||_{H^3(\mathbb{R}_+^3)}$. In addition, $r$ does not depend on the choice of $(t_0,x_0)$; in particular, $\theta$ is uniformly $C^r$ throughout the interval $[0,T]\times\mathbb{R}^2$, so the lemma is complete.  
\end{proof}

\section{Bootstrapping}

We now show that $\dnu_1(t,x) \in L^\infty([0,T];\Bdot_{\infty,\infty}^{1}(\mathbb{R}^2))$, which will give that $\grad^\perp \Psi_1 \in L^\infty([0,T]\times[0,\infty);\Bdot_{\infty,\infty}^{1}(\mathbb{R}^2))$. Here $[0,T]\times [0,\infty)$  denotes $t$ and $z$, and $\mathbb{R}^2$ includes points $x=(x_1,x_2)$ belonging to flat planes $z=z_0$.  Due to the fact that the Poisson kernel is the fundamental solution of the equation
\begin{equation}\label{Poisson}
\partial_t \theta + \half \theta = 0
\end{equation}
we need the following two lemmas. These lemmas provide estimates on the regularity of the solution to an     inhomogeneous version of \cref{Poisson}.  Let us remark that in the case of critical SQG, one can use either potential theory in the style of \cite{cv} or Littlewood-Paley arguments in the style of \cite{cw} to bootstrap the regularity.  Also, both potential theory and Littlewood-Paley arguments can be used to show the sharp $\Bdot^1_{\infty,\infty}$ bound coming from the forcing term.  However, the most direct method in our situation seems to be to use potential theory for the nonlinear term and Littlewood-Paley arguments for the forcing.    

Roughly speaking, the following lemma will show that the regularity of the nonlinear terms is additive;  if $\theta$ is H\"{o}lder continuous in space-time with exponent $\alpha_1$ and $u$ in space with exponent $\alpha_2$, then the convolution of their product with the Poisson kernel in space-time is H\"{o}lder continuous in space with exponent $\alpha_1+ \alpha_2$.  Let us give an intuition as to why such a statement should hold.  Given functions $f\in C^{\alpha_1}$, $g\in C^{\alpha_2}$, then $fg\in C^{\alpha_1 \wedge \alpha_2}$.  However, if $f(x_0)=g(x_0)=0$, then 
$$ |f(x)g(x) -f(x_0)g(x_0)| = |(f(x)-f(x_0))(g(x)-g(x_0))|\leq |x-x_0|^\alpha, $$
and so $fg$ is $C^\alpha$ at $x_0$.  If we are trying to increase the regularity at $(t_0,x_0)$, we can ensure that $\theta(t_0,x_0)=u(t,x_0)=0$ for all $t\in[0,t_0]$ by performing a change of variables which follows the characteristics.  Then the nonlinear term is effectively $C^\alpha$, allowing us to bootstrap the regularity up to the space $\Bdot_{\infty,\infty}^1$.  

\begin{lemma}\label{nonlinear}
Let $f(t,x) \in C^{\alpha_1}([0,t_0] \times \mathbb{R}^2)$, $h(t,x) \in L^\infty([0,t_0];C^{\alpha_2}(\mathbb{R}^2))$, and $ f(t_0,0) = h(t,0) = 0 $ for all $t\in[0,t_0]$.  Let $\mathcal{P}(t,x)$ be the Poisson kernel (extended to equal $0$ for $t$ negative).  Let $\alpha=\alpha_1+\alpha_2$, and define
$$ g(t,x) = \int_0^t \int_{\mathbb{R}^2} \mathcal{P}(t-s,x-y) \dive (h(s,y) f(s,y)) \,dy\,ds . $$  
\begin{enumerate}
\item If $0<\alpha<1$, then $$ \sup_{x\in\mathbb{R}^2} \frac{|g(t_0,x) - g(t_0,0)|}{|x|^{\alpha}} < C  || f ||_{C^{\alpha_1}} ||h||_{L^\infty(C^{\alpha_2})} .$$
\item If $1 \leq \alpha <2$, then $$ \sup_{x\in\mathbb{R}^2} \frac{|g(t_0,x) - 2g(t_0,0) + g(t_0,-x)|}{|x|^{\alpha}} < C  || f ||_{C^{\alpha_1}} ||h||_{L^\infty(C^{\alpha_2})} .$$
\end{enumerate}
\end{lemma}

\begin{proof}
Before starting, we remark that the gradient of the Poisson kernel 
$$ \nabla_x \mathcal{P}(t,x) = C \frac{tx}{(|x|^2+t^2)^\frac{5}{2}} $$
is homogenous of degree 3 in $(-\infty,\infty)\times\mathbb{R}^2$.  Using the fact that it is smooth away from the origin and has mean value zero in space over any set $\{t\} \times B(r,0) $, we see that $\nabla_x \mathcal{P}$ is a singular integral in space-time. 
Beginning with the first case, we integrate by parts and split the integral around the singularity to obtain 
\begin{align*}
 g(t_0,0) - g(t_0,x) &= \int_0^{t_0} \int_{\mathbb{R}^2} ( \mathcal{P}(t_0-s,-y) - \mathcal{P}(t_0-s, x-y) ) \dive (h(s,y) f(s,y)) \,dy\,ds \\
 &= - \int_0^{t_0} \int_{\mathbb{R}^2} ( \nabla_x \mathcal{P}(t_0-s,-y) - \nabla_x\mathcal{P}(t_0-s, x-y) )  h(s,y) f(s,y) \,dy\,ds \\
 &= - \iint_{B_{3|x|}(t_0,0)}  \nabla_x\mathcal{P}(t_0-s,-y)  h(s,y) f(s,y) \,dy\,ds \\
 &\qquad  + \iint_{B_{3|x|}(t_0,0)}  \nabla_x\mathcal{P}(t_0-s,x-y)  h(s,y) f(s,y) \,dy\,ds \\
 &\qquad - \iint_{(B_{3|x|}(t_0,0))^\complement} ( \nabla_x\mathcal{P}(t_0-s,-y) - \nabla_x\mathcal{P}(t_0-s, x-y) ) h(s,y) f(s,y) \,dy\,ds \\
 &= I + II + III
\end{align*}
We start with $I$; using the fact that $f(t_0,0) = h(s,0)=0$, we integrate in polar coordinates in space-time to obtain 
\begin{align*}
I &= - \iint_{B_{3|x|}(t_0,0)}  \nabla_x\mathcal{P}(t_0-s,-y)  (h(s,y)-h(s,0)) (f(s,y)-f(t_0,0)) \,dy\,ds \\
  &\leq C ||h||_{L^\infty(C^{\alpha_2})}||f||_{C^{\alpha_1}} \int_0^{3|x|} r^{\alpha_1+\alpha_2-1} \,dr\\
  &= C||h||_{L^\infty(C^{\alpha_2})}||f||_{C^{\alpha_1}}|x|^{\alpha}.
\end{align*}

Moving to $II$, note that by the mean value condition on $\nabla_x \mathcal{P}$ and the assumptions on $u$ and $\theta$, 
\begin{align*}
II = \iint_{B_{3|x|}(t_0,0)}  &\nabla_x\mathcal{P}(t_0-s,x-y)  h(s,y) f(s,y) \,dy\,ds  \\
&=\iint_{B_{3|x|}(t_0,0)} \nabla_x\mathcal{P}(t_0-s,x-y) \left[ (h(s,y)-h(s,0))(f(s,y)-f(t_0,x)) \right]\,dy\,ds\\
&+ \iint_{B_{3|x|}(t_0,0)} \nabla_x\mathcal{P}(t_0-s,x-y) \left[ (h(s,y)-h(s,x))(f(t_0,x)-f(t_0,0)) \right]\,dy\,ds\\
& \leq ||h||_{L^\infty(C^{\alpha_2})} |x|^{\alpha_2} ||f||_{C^{\alpha_1}} \int_0^{3|x|} r^{\alpha_1-1} dr + ||f||_{C^{\alpha_1}} |x|^{\alpha_1} ||h||_{L^\infty(C^{\alpha_2})} \int_0^{3|x|} r^{\alpha_2-1} dr \\
& \leq C||h||_{L^\infty(C^{\alpha_2})}||f||_{C^{\alpha_1}}|x|^{\alpha}.
\end{align*}
Finally, since the domain of integration for $III$ is a fixed distance away from the singularity, a first order space-time Taylor estimate on $\nabla_x \mathcal{P}$ gives that on the domain of integration,
$$ |\mathcal{P}(t_0-s,x-y)-\mathcal{P}(t_0-s, -y)| \leq \frac{|x|}{|(t_0-s,x-y)|^4} .$$
Therefore, using the properties of $u$, $\theta$, and $\nabla_x \mathcal{P}$ gives
\begin{align*}
III &\leq \iint_{(B_{3|x|}(t_0,0))^\complement}  \big{|} ( \nabla_x\mathcal{P}(t_0-s,-y) - \nabla_x\mathcal{P}(t_0-s, x-y) ) \cdot \\
    &\qquad \qquad (h(s,y)-h(s,0) (f(s,y)-f(t_0,0) \big{|} \,dy\,ds \\
    &\leq ||h||_{L^\infty(C^{\alpha_2})}||f||_{C^{\alpha_1}} \int_{3|x|}^{\infty} \frac{|x|^{\alpha+1}}{r^2} \,dr \\
    &\leq C ||h||_{L^\infty(C^{\alpha_2})}||f||_{C^{\alpha_1}}|x|^{\alpha}
\end{align*}
Combining estimates for $I$, $II$, and $III$ gives the result.  

We now consider the case $1\leq \alpha < 2$.  As before, we integrate by parts and split the integral into two pieces;  
\begin{align*}
g(t_0,x) - 2g(t_0,0) + g(t_0,-x) &=  \int_0^{t_0}\int_{\mathbb{R}^2} ( \mathcal{P}(t_0-s,x-y) - 2\mathcal{P}(t_0-s, -y) \\ 
&\qquad +\mathcal{P}(t_0-s,-x-y) ) \cdot \dive (h(s,y) f(s,y)) \,dy\,ds  \\
&= - \iint_{B(3|x|,0)} ( \nabla_x \mathcal{P}(t_0-s,x-y) - 2 \nabla_x \mathcal{P}(t_0-s, -y) \\ 
&\qquad \qquad + \nabla_x \mathcal{P}(t_0-s,-x-y) ) \cdot h(s,y) f(s,y) \,dy\,ds  \\
&\quad - \iint_{(B(3|x|,0))^\complement} ( \nabla_x \mathcal{P}(t_0-s,x-y) - 2\nabla_x \mathcal{P}(t_0-s, -y) \\ 
&\qquad \qquad +\nabla_x \mathcal{P}(t_0-s,-x-y) ) \cdot  h(s,y) f(s,y) \,dy\,ds  \\
&= I + II
\end{align*}
For the first piece, noticing that 
$$ g(t_0,x) - 2g(t_0,0) + g(t_0,-x) = (g(t_0,x)-g(t_0,0)) - (g(t_0,0)-g(t_0,-x)) $$
we can use the local estimate from the first part to conclude that $I\leq C||h||_{L^\infty(C^{\alpha_2})}||f||_{C^{\alpha_1}}|x|^{\alpha}$.  

For $II$, we can use the fact that 
$$ \nabla_x \mathcal{P}(t_0-s,x-y) - 2\nabla_x \mathcal{P}(t_0-s, -y) + \nabla_x \mathcal{P}(t_0-s,-x-y) $$
vanishes to first order.  Since the domain of integration in $II$ avoids the singularity, a second order space-time Taylor expansion gives that in the domain of integration,
$$ | \mathcal{P}(t_0-s,x-y) - 2\mathcal{P}(t_0-s, -y) +\mathcal{P}(t_0-s,-x-y) | \leq \frac{|x|^2}{|(t_0-s, -y)|^5}.$$
Therefore, we can write 
\begin{align*}
II &\leq \iint_{(B_{3|x|}(t_0,0))^\complement}  \big{|} ( \nabla_x\mathcal{P}(t_0-s,x-y) - 2\nabla_x\mathcal{P}(t_0-s, -y) +\nabla_x\mathcal{P}(t_0-s, -x-y) ) \cdot \\
    &\qquad \qquad (h(s,y)-h(s,0) (f(s,y)-f(t_0,0) \big{|} \,dy\,ds \\
    &\leq ||h||_{L^\infty(C^{\alpha_2})}||f||_{C^{\alpha_1}}\int_{3|x|}^{\infty} \frac{|x|^{\alpha+2}}{r^3} \,dr \\
    &\leq C||h||_{L^\infty(C^{\alpha_2})}||f||_{C^{\alpha_1}} |x|^{\alpha}
\end{align*}
concluding the proof of the second part.

\end{proof}

We provide now a short proof of the estimate needed for the right hand side.  

\begin{lemma}\label{rhs}
Let $\omega \in L^\infty([0,T];\Bdot^1_{\infty,\infty}(\mathbb{R}^2)) $, and define for $t\in[0,T]$
$$ g(t,x) = \int_0^t \int_{\mathbb{R}^2} \mathcal{P}(t-s,x-y) \dive (\omega(s,y)) \,dy \,ds .$$
Then $g \in L^\infty([0,T];\Bdot^1_{\infty,\infty}(\mathbb{R}^2))$.  
\end{lemma}
\begin{proof}
We must show that $\sup_j 2^j ||\Delta_j g||_{L^\infty} < \infty$.  Recall that $\Delta_j$ is a dilation in frequency by a factor of $2^j$ of a Fourier multiplier which isolates frequences on an annulus of radius $1$. We let $\tilde{\Delta}_j$ be a dilation by a factor of $2^j$ of a Fourier multiplier which strictly contains the annulus of radius $1$, ensuring that the frequency support of $\Delta_j$ is contained inside that of $\tilde{\Delta}_j$.  Then we can write
\begin{align*}
\Delta_j g(t,x) &= \int_0^t \int_{\mathbb{R}^2} \mathcal{P}(t-s,x-y) \dive(\Delta_j \omega(s,y)) \,dy \,ds \\
                &= \int_0^t \int_{\mathbb{R}^2} \nabla_x \tilde{\Delta}_j \mathcal{P}(t-s,x-y) \Delta_j \omega(s,y) \,dy \,ds \\
                &\leq C \int_0^t 2^j e^{-(t-s)2^j} 2^{-j} || \omega(s,\cdot) ||_{\Bdot^1_{\infty,\infty}} \,ds\\
                &\leq C 2^{-j} ||\omega||_{L^\infty(\Bdot^1_{\infty,\infty})} \\
\end{align*}
\end{proof}

We can now show that the regularity of $\dnu$ can be bootstrapped all the way to $\Bdot^1_{\infty,\infty}$.  Let $\Psi$ be a strong solution to (QG) on $[0,T]$.  We have that $\dnu=\theta$ satisfies
$$ \partial_t \theta  + \half \theta =- u \cdot \grad \theta + \lap \Psi_2. $$
From \cref{holder}, we have that $\theta\in C^r([0,T]\times \mathbb{R}^2)$.  From the Riesz transform, we have also that $\grad^\perp\Psi_1 |_{z=0} \in L^\infty([0,T]; C^r(\mathbb{R}^2))$.  By interpolating (4) and (7) from \cref{apriori}, we have that for all $\alpha<1$, $\nabla^\perp\Psi_2 |_{z=0} \in L^\infty([0,T]; C^\alpha(\mathbb{R}^2)) \cap L^\infty([0,T]; \Bdot^1_{\infty,\infty}(\mathbb{R}^2))$.  We can combine \cref{nonlinear} with \cref{rhs} to show that $\theta\in L^\infty(\Bdot^1_{\infty,\infty})$.  In order to apply \cref{nonlinear}, we fix $(t_0,x_0)$ and perform a change of variables which follows the flow.  Specifically, let 
$$  \left\{
       \begin{array}{@{}l@{\thinspace}l}
       \dot{\Gamma}(t) = u(t,\Gamma(t))\\
       \Gamma(t_0) = x_0\
       \end{array}  \right.  $$
This trajectory is well-defined since we are on the interval for which (QG) has a smooth solution.  Crucially, the argument relies only on the existence of $\Gamma(t)$ and the boundedness of $\dot{\Gamma}(t)$, not the smoothness.  Define 
$$ \tilde{\theta}(t,x) = \theta(t,x+\Gamma(t)) - \theta(t_0,x_0)$$
$$ \tilde{u}(x,t) = u(t,x+\Gamma(t)) $$
$$ \lap \tilde{\Psi}_2(t,x) = \lap \Psi_2(t,x+\Gamma(t)) .$$
Then $\tilde{\theta}$ solves the equation 
$$ \partial_t \tilde{\theta}(t,x)  + \half \tilde{\theta}(t,x) = - (\tilde{u}(t,x)-\tilde{u}(t,0))\cdot \grad \tilde{\theta}(t,x) + \lap \tilde{\Psi}_2(t,x) .$$
 The norms of $\tilde{\theta}\in C^r([0,T]\times\mathbb{R}^2)$, $\tilde{u}\in L^\infty([0,T];C^r(\mathbb{R}^2))$, and $\nabla^\perp\tilde{\Psi}_2 |_{z=0} \in L^\infty([0,T]; C^\alpha(\mathbb{R}^2)) \cap L^\infty([0,T]; \Bdot^1_{\infty,\infty}(\mathbb{R}^2))$ are preserved under this change of variables since $\dot{\Gamma}(t)$ is bounded.  We split $\tilde{\theta} = g_0 + g_1 + g_2$, where
$$ g_0(t,x) = \tilde{\theta}(0,\cdot) \ast \mathcal{P}_t(\cdot) (x) $$
$$ \partial_t g_1 + \half g_1 =  - (\tilde{u}-\tilde{u}(t,0))\cdot \grad \tilde{\theta}$$
$$ \partial_t g_2 + \half g_2 = \lap \tilde{\Psi}_2 .$$
Since $g_0$ is a convolution with the Poisson kernel of a shifted version of $\tilde{\theta}$, its regularity depends only on that of the initial data. Focusing on the other two terms, we have that $g_1$ can be written using Duhamel's formula with $f(t,x)=\tilde{\theta}(t,x)$ and $h(t,x)=\tilde{u}(t,x)-\tilde{u}(t,0)$, satisfying the assumptions of \cref{nonlinear}.  Therefore, $g_1$ is $C^{2r}$ in space at $(t_0,x_0)$.  In addition, $g_2$ can also be written using Duhamel's formula with $\omega = \grad \tilde{\Psi_2}$, satisfying the assumptions of \cref{rhs}, and so $g_2 \in L^\infty(\Bdot^1_{\infty,\infty})$.  Repeating the argument for arbitrary $(t_0,x_0)$ and recalling that the difference quotient characterization of $\Bdot^1_{\infty,\infty}$ is locally stronger than $C^{2r}$ for any $2r<1$ shows that $\theta \in C^r([0,T]\times\mathbb{R}^2) \cap L^\infty([0,T]; C^{2r}(\mathbb{R}^2))$.  Applying the Riesz transform combined with \cref{corollaries} and \cref{equivalences} shows that $\grad^\perp \Psi_1|_{z=0} \in L^\infty([0,T]; C^{2r}(\mathbb{R}^2))$.  Recalling the a priori estimates in parts (4) and (7) of \cref{apriori}, we have also that $\nabla\Psi_2$, and therefore $u$, are in $L^\infty([0,T]; C^{2r}(\mathbb{R}^2))$.  We then repeat the argument $N$ times, for $Nr\geq 1$.  On the last iteration, $g_0$ and $g_1$ become $C^{1,Nr-1}$; however, the regularity of $g_2$ becomes the limiting factor, since $g_2 \in L^\infty([0,T];\Bdot_{\infty,\infty}^1(\mathbb{R}^2))$.  We cannot bootstrap any higher, and thus we have shown that $\theta\in L^\infty([0,T];\Bdot_{\infty,\infty}^1(\mathbb{R}^2))$.  

	We now show that for any $z$, $\nabla \Psi_1(\cdot, z)$ enjoys the same regularity in $x$ as $\dnu_1$.  Recalling that the $L^1(\mathbb{R}^2)$ norm of the Poisson kernel $\mathcal{P}_z(x):=\mathcal{P}(x,z)$ is equal to 1 for any $z$, we can say that for all $j$,
$$ ||\Delta_j \left(\mathcal{P}_z \ast (\dnu_1) \right) ||_{L^\infty(\mathbb{R}^2)} \leq || \Delta_j \left(\dnu_1\right)||_{L^\infty(\mathbb{R}^2)} $$
(where the Littlewood-Paley projection is in $x$ only).
This shows that $\left(\mathcal{P}_z \ast (\dnu_1)\right)  \in \Bdot_{\infty,\infty}^1(\mathbb{R}^2)$ with norm less than or equal to that of $\dnu_1$.  Furthermore, this estimate is uniform in $z$. Next, we note that 
$$ \nabla\Psi_1(z,x) = \big{(}\mathcal{P}_z\ast(\dnu_1)(x), \mathcal{R}_1(\mathcal{P}_z\ast(\dnu_1))(x), \mathcal{R}_2(\mathcal{P}_z\ast(\dnu_1))(x) \big{)} $$
where $\mathcal{R}_i$ is the $i^{th}$ Riesz transform.  Using the boundedness of the Riesz transforms on Besov spaces (part (1) of \cref{corollaries}) and the above observations regarding the Poisson kernel, we have that $\nabla\Psi_1 \in L^\infty([0,T]\times[0,\infty);\Bdot_{\infty,\infty}^{1}(\mathbb{R}^2))$.  Recalling (4) of \cref{apriori}, which gives that $\nabla\Psi_2 \in L^\infty([0,T]\times[0,\infty);\Bdot_{\infty,\infty}^{1}(\mathbb{R}^2))$, we have shown the following:
\begin{theorem}\label{velocity}
Let $\Psi$ be a strong solution to $(\QG)$ on $[0,T]$; then there exists $C$ depending only on $||\Psi_0||_{H^3(\mathbb{R}_+^3)}$ such that $\Psi$ satisfies 
$$\nabla \Psi \in L^\infty([0,T]\times[0,\infty);\Bdot_{\infty,\infty}^{1}(\mathbb{R}^2))$$
with norm less than or equal to $C$.
\end{theorem}

\section{Propagation of Regularity}

We begin by using the transport equations on both $\nabla\Psi$ and $\Delta\Psi$ to show that smoothness in the flat variable $x=(x_1,x_2)$ is propagated in time.  Then, using this result in conjunction with the stratification of the flow will show that smoothness in all variables is propagated in time.  Since the local existence theorem gives existence of strong solutions on a time interval which depends only on $|| \nabla\Psi_0 ||_{H^3(\mathbb{R}_+^3)}$, obtaining a differential inequality which bounds $|| \nabla\Psi(t)||_{H^3(\mathbb{R}_+^3)}$ in time allows us to apply a continuation principle, thus showing that solutions are smooth for all time. We work again on a time interval for which $\nabla\Psi$ is a solution to (QG), justifying the calculations.  

\begin{lemma}\label{flat}
For any $T>0$, $R>0$, there exists a constant $C_{T,R}$ such that the following is true. Let $\nabla\Psi\in L^\infty([0,t_0];H^3(\mathbb{R}_+^3))$ be a solution to $(\QG)$ for all $t_0 < T$.  If $||\nabla\Psi_0||_{H^{s+1}(\mathbb{R}_+^3)}<R$, then for all $t<T$, $$ ||\grad^{s+1}(\nabla\Psi)(t)||_{L^2(\mathbb{R}_+^3)} + ||\grad^s(\Delta\Psi)(t)||_{L^2(\mathbb{R}_+^3)} \leq C_{T,R}.$$ 
\end{lemma}

\begin{proof}
Recall that for $s=|\alpha|$, \cref{commutator} gives the commutator estimate 
$$ \left|\left|D^\alpha (fg)-f D^\alpha g \right|\right|_{L^2} \leq C(s) \left(||\nabla f||_{L^\infty}||\nabla^{(s-1)}g||_{L^2} + ||g||_{L^\infty}||\nabla^s f||_{L^2}\right) .$$
Also recall that for $h=\grad H$, \cref{inequality} provides the bound 
$$ ||h||_{L^\infty} \leq C||H||_{L^\infty} + {C}||h||_{\Bdot_{\infty,\infty}^0}\left( 1+\log{\frac{||h||_{\Hdot^\frac{3}{2}}}{||h||_{\Bdot_{\infty,\infty}^0}}} \right) .$$
Using the fact that $\partial_{zz} \Psi = \Delta\Psi - \lap\Psi$ and applying \cref{trace} with $u=\grad^2(\nabla\Psi)$ gives that  
\begin{equation}\label{eq:trace}
\sup_z ||\nabla\Psi(z,\cdot)||_{\Hdot^\frac{5}{2}} = \sup_z ||\grad^2(\nabla\Psi)(z,\cdot)||_{\Hdot^\frac{1}{2}(\mathbb{R}^2)} \leq ||\grad^3 (\nabla\Psi)||_{L^2(\mathbb{R}_+^3)} + ||\grad^2 (\Delta\Psi) ||_{L^2(\mathbb{R}_+^3)} . 
\end{equation}
From \cref{velocity}, $\nabla\Psi \in L^\infty([0,t_0]\times[0,\infty);\Bdot_{\infty,\infty}^{1}(\mathbb{R}^2))$. We have that $\nabla\Psi\in L^\infty(\mathbb{R}_+^3) $.  Then applying \cref{inequality} to $h=\grad(\nabla\Psi)$, \cref{apriori}, \cref{velocity}, and \eqref{eq:trace}, we obtain the following:
\begin{align}
||\grad(\nabla&\Psi)||_{L^\infty(\mathbb{R}_+^3)} = \sup_z ||\grad(\nabla\Psi)(z,\cdot)||_{L^\infty(\mathbb{R}^2)}\nonumber \\ 
&\leq C \sup_z \left( ||\nabla \Psi(z,\cdot)||_{L^\infty} + ||\grad(\nabla \Psi)(z,\cdot)||_{\Bdot_{\infty,\infty}^0}\left( 1+\log\frac{||\grad(\nabla \Psi)(z,\cdot)||_{\Hdot^\frac{3}{2}}}{||\grad(\nabla\Psi)(z,\cdot)||_{\Bdot_{\infty,\infty}^0}} \right)\right)\nonumber \\ 
&\leq C \sup_z \left( 1 + ||\grad(\nabla \Psi)(z,\cdot)||_{\Bdot_{\infty,\infty}^0} \left( \log||\nabla\Psi(z,\cdot)||_{\Hdot^\frac{5}{2}} -\log{||\grad(\nabla\Psi)(z,\cdot)||_{\Bdot_{\infty,\infty}^0}} \right)\right) \nonumber\\ 
&\leq C \sup_z \left( 1 + \logplus||\nabla\Psi(z,\cdot)||_{\Hdot^\frac{5}{2}} \right)\nonumber\\ 
&\leq C \left( 1+\logplus{\left(||\grad^3 (\nabla\Psi)||_{L^2(\mathbb{R}_+^3)} + ||\grad^2 (\Delta\Psi) ||_{L^2(\mathbb{R}_+^3)}\right)} \right) . \label{log}
\end{align}

We shall obtain a differential inequality from the transport equations on $\nabla\Psi$ and $\Delta\Psi$.  Beginning with the former, we have from \cref{Pgradient} that
$$ \partial_t(\nabla \Psi)+ \mathbb{P}_\nabla(\grad^\perp \Psi \cdot \grad(\nabla \Psi)) = \nabla F .$$ 
We shall apply the commutator bound by putting $f=\grad^\perp \Psi$, $g=\grad(\nabla\Psi)$, and applying a differential operator $\Dbar^\alpha$ with $|\alpha|=s+1$. Using \eqref{log} and the fact that $|s|\geq 2$ , we have
\begin{align*}
\big{|}\big{|}\big{[}\grad^\perp \Psi, \Dbar^\alpha\big{]}(&\grad (\nabla \Psi)(z,\cdot))\big{|}\big{|}_{L^2(\mathbb{R}^2)}\leq
C \bigg{(}||\grad(\grad^\perp\Psi)(z,\cdot)||_{L^\infty} ||\grad^{s}(\grad(\nabla\Psi))(z,\cdot)||_{L^2}\nonumber\\
&\qquad\qquad\qquad\qquad\qquad+ ||\grad(\nabla\Psi)(z,\cdot)||_{L^\infty} ||\grad^{s+1}(\grad^\perp\Psi)(z,\cdot)||_{L^2} \bigg{)} \nonumber\\
&\leq C ||\grad^{s+1}(\nabla\Psi)(z,\cdot)||_{L^2}||\grad(\nabla\Psi)(z,\cdot)||_{L^\infty} \nonumber\\
&\leq C ||\grad^{s+1}(\nabla\Psi)(z,\cdot)||_{L^2}  \left( 1+\logplus{\left(||\grad^3 (\nabla\Psi)||_{L^2(\mathbb{R}_+^3)} + ||\grad^2 (\Delta\Psi) ||_{L^2(\mathbb{R}_+^3)}\right)} \right) . \nonumber\\
\end{align*}
Applying the differential operator $\Dbar^\alpha$ with $|\alpha|=s+1\geq 3$, multiplying by $\Dbar^\alpha \nabla \Psi$, integrating by parts, and utilizing the commutator estimate gives
\begin{align*}
\frac{1}{2}\frac{\partial}{\partial t} \int_{\mathbb{R}_+^3} | &\Dbar^\alpha \nabla \Psi|^2 = \int_{\mathbb{R}_+^3} \mathbb{P}_\nabla \left[ [\grad^\perp \Psi, \Dbar^\alpha](\grad (\nabla \Psi)) \right]\cdot \nabla \Dbar^\alpha \Psi + \int_{\mathbb{R}_+^3}{\nabla \Dbar^\alpha F \cdot \nabla \Dbar^\alpha \Psi }\\
&= \int_{\mathbb{R}_+^3} \left[\grad^\perp \Psi, \Dbar^\alpha\right](\grad (\nabla \Psi)) \cdot \nabla \Dbar^\alpha \Psi  + \int_{\mathbb{R}^2} (\Dbar^\alpha (\partial_\nu F))(\Dbar^\alpha \Psi)  \\
&= \int_{\mathbb{R}_+^3} \left[\grad^\perp \Psi, \Dbar^\alpha\right](\grad(\nabla \Psi)) \cdot \nabla \Dbar^\alpha \Psi + \int_{\mathbb{R}^2} (\Dbar^\alpha (\lap\Psi))(\Dbar^\alpha \Psi) \\
&\leq \int_0^\infty \int_{\mathbb{R}^2} \left[\grad^\perp \Psi(z,\cdot), \Dbar^\alpha\right](\grad(\nabla \Psi)(z,\cdot)) \cdot \nabla \Dbar^\alpha \Psi(z,\cdot)\,dx\,dz \\
&\leq  \int_0^\infty \left|\left|\left[\grad^\perp \Psi(z,\cdot), \Dbar^\alpha\right] (\grad(\nabla \Psi)(z,\cdot)) \right|\right|_{L^2(\mathbb{R}^2)} \left|\left| \Dbar^\alpha \nabla \Psi(z,\cdot) \right|\right|_{L^2(\mathbb{R}^2)} \,dz \\
&\leq C \int_0^\infty \left|\left|\grad^{s+1}(\nabla\Psi)(z,\cdot)\right|\right|_{L^2(\mathbb{R}^2)}^2  \left( 1+\logplus\left({||\grad^3 (\nabla\Psi)||_{L^2(\mathbb{R}_+^3)} + ||\grad^2 (\Delta\Psi) ||_{L^2(\mathbb{R}_+^3)}}\right) \right) \,dz \\
&\leq C \left|\left|\grad^{s+1}(\nabla\Psi)\right|\right|_{L^2(\mathbb{R}_+^3)}^2  \left( 1+\logplus{\left(||\grad^3 (\nabla\Psi)||_{L^2(\mathbb{R}_+^3)} + ||\grad^2 (\Delta\Psi) ||_{L^2(\mathbb{R}_+^3)}\right)} \right).
\end{align*}

We now move to the transport equation on $\Delta\Psi$:
$$ \partial_t(\Delta\Psi)+\grad^\perp\Psi\cdot\grad(\Delta\Psi)=0 .$$
We shall apply the commutator bound by putting $f=\grad^\perp \Psi$, $g=\Delta\Psi$, and applying a differential operator $\Dbar^\alpha$ with $|\alpha|=s$. Using the $L^\infty$ bound on $\Delta\Psi$, \eqref{log}, and the fact that $|s|\geq 2$ , we have
\begin{align*}
\big{|}\big{|}\big{[}\grad^\perp \Psi, \Dbar^\alpha\grad\cdot\big{]}(&\Delta \Psi)(z,\cdot)\big{|}\big{|}_{L^2(\mathbb{R}^2)}\leq C \bigg{(}||\grad(\grad^\perp\Psi)(z,\cdot)||_{L^\infty} ||\grad^{s}(\Delta\Psi)(z,\cdot)||_{L^2}\nonumber\\
&\qquad\qquad\qquad\qquad+ ||\Delta\Psi(z,\cdot)||_{L^\infty} ||\grad^{s+1}(\grad^\perp\Psi)(z,\cdot)||_{L^2} \bigg{)} \nonumber\\
&\leq C \left(||\grad^{s}(\Delta\Psi)(z,\cdot)||_{L^2}||\grad(\nabla\Psi)(z,\cdot)||_{L^\infty} + ||\grad^{s+1}(\grad^\perp\Psi)(z,\cdot)||_{L^2}  \right)\nonumber\\
&\leq C \left(||\grad^{s}(\Delta\Psi)(z,\cdot)||_{L^2} +||\grad^{s+1}(\nabla\Psi)(z,\cdot)||_{L^2} \right)\\
&\qquad\times\left( 1+\logplus{\left(||\grad^3 (\nabla\Psi)||_{L^2(\mathbb{R}_+^3)} + ||\grad^2 (\Delta\Psi) ||_{L^2(\mathbb{R}_+^3)}\right)} \right) . \nonumber\\
\end{align*}

Applying the differential operator $\Dbar^\alpha$ with $|\alpha|=s\geq 2$, multiplying by $\Dbar^\alpha \Delta \Psi$, integrating by parts, and utilizing the commutator estimate gives
\begin{align*}
\frac{1}{2}\frac{\partial}{\partial t} \int_{\mathbb{R}_+^3} |\Dbar^\alpha &\Delta \Psi|^2 = \int_{\mathbb{R}_+^3} \left[\grad^\perp \Psi, \Dbar^\alpha\grad\cdot\right] (\Delta \Psi) \cdot \Dbar^\alpha \Delta \Psi \\
&= \int_0^\infty \int_{\mathbb{R}^2} \left[\grad^\perp \Psi(z,x), \Dbar^\alpha\grad\cdot\right](\Delta \Psi)(z,x) \cdot \Dbar^\alpha \Delta\Psi(z,x)\,dx\,dz \\
&\leq  \int_0^\infty \left|\left|\left[\grad^\perp \Psi(z,\cdot), \Dbar^\alpha\grad\cdot\right] (\Delta \Psi)(z,\cdot) \right|\right|_{L^2(\mathbb{R}^2)} \left|\left| \Dbar^\alpha \Delta \Psi(z,\cdot) \right|\right|_{L^2(\mathbb{R}^2)} \,dz \\
&\leq C \int_0^\infty \left( ||\grad^s(\Delta\Psi)(z,\cdot)||_{L^2(\mathbb{R}^2)}^2 + ||\grad^{s+1}(\nabla\Psi)(z,\cdot)||_{L^2(\mathbb{R}^2)}||\grad^s(\Delta\Psi)(z,\cdot)||_{L^2(\mathbb{R}^2)} \right)\\
&\qquad\qquad\times\left( 1+\logplus{||\grad^3 (\nabla\Psi)||_{L^2(\mathbb{R}_+^3)} + ||\grad^2 (\Delta\Psi) ||_{L^2(\mathbb{R}_+^3)}} \right) \,dz \\
&\leq C \left( \left|\left|\grad^s(\Delta\Psi)\right|\right|_{L^2(\mathbb{R}_+^3)}^2 + ||\grad^{s+1}(\nabla\Psi)||_{L^2(\mathbb{R}_+^3)} ||\grad^s(\Delta\Psi)||_{L^2(\mathbb{R}_+^3)} \right)\\  
&\qquad\times\left( 1+\logplus{\left(||\grad^3 (\Delta\Psi)||_{L^2(\mathbb{R}_+^3)} + ||\grad^2 (\Delta\Psi) ||_{L^2(\mathbb{R}_+^3)}\right)} \right) .
\end{align*}
Therefore, we can sum over $\alpha$ in both inequalities and apply Gronwall's inequality to the sum 
$$ ||\grad^{s+1} (\Delta\Psi)||_{L^2(\mathbb{R}_+^3)} + ||\grad^s (\Delta\Psi) ||_{L^2(\mathbb{R}_+^3)} ,$$
finishing the proof.
\end{proof}

We now show that regularity in $z$ can be propagated as well.

\begin{theorem}\label{full}
For any $T>0$, $R>0$, there exists a constant $C_{T,R}$ such that the following is true. Let $\nabla\Psi\in L^\infty([0,t_0];H^3(\mathbb{R}_+^3))$ be a solution to $(\QG)$ for all $t_0 < T$.  If $||\nabla\Psi_0||_{H^s(\mathbb{R}_+^3)}<R$, then for all $t < T$, $$||\nabla\Psi(t)||_{H^{s}(\mathbb{R}_+^3)} \leq C_{T,R}.$$ 
\end{theorem}

\begin{proof}
From \cref{flat}, Sobolev embedding, and the trace estimate, $||\grad(\nabla\Psi)(t)||_{L^\infty(\mathbb{R}_+^3)}$ is bounded.  Also, observe that using the identity $\partial_{zz} = \Delta - \lap$, we have that 
$$ ||\nabla^{s}(\nabla\Psi) ||_{L^2} \leq C \left(||\nabla^{s-1}(\Delta \Psi)||_{L^2} + ||\grad^{s}(\nabla\Psi)||_{L^2}\right) .$$
By \cref{flat}, we have that $||\grad^{s}(\nabla\Psi)||_{L^2}<\infty$.  Thus the theorem will be shown if $\Delta\Psi\in H^{s-1}$ for all time.  Applying a differential operator $D^\alpha$ with $|\alpha|=s-1\geq 2$, multiplying by $D^\alpha \Delta \Psi$, integrating by parts, and using the commutator estimate (in $\mathbb{R}_+^3$) in conjunction with the above observations, we have
\begin{align*}
\frac{1}{2}\frac{\partial}{\partial t} \int_{\mathbb{R}_+^3} |D^\alpha \Delta \Psi|^2 &= \int_{\mathbb{R}_+^3} \left[\grad^\perp \Psi, D^\alpha \grad \cdot\right] (\Delta \Psi) \cdot D^\alpha \Delta \Psi \\
&\leq C \left( ||\grad(\nabla\Psi)||_{L^\infty}||\nabla^{s-1}(\Delta\Psi)||_{L^2} + ||\Delta\Psi||_{L^\infty}||\nabla^{s}(\nabla\Psi)||_{L^2} \right) ||\nabla^{s-1} (\Delta\Psi) ||_{L^2 }\\
&\leq C (||\nabla^{s-1} (\Delta\Psi) ||_{L^2}^2 + ||\nabla^{s-1} (\Delta\Psi) ||_{L^2 })
\end{align*}
Summing over $\alpha$ and applying Gronwall's inequality now finishes the proof.  
\end{proof}

\begin{proof}[Proof of \cref{main}]
Applying a continuation principle in conjunction with \cref{full} gives the first part of \cref{main}; namely, if $\nabla\Psi_0\in H^s(\mathbb{R}_+^3)$ for some $s\geq3$, then for all $T>0$, there exists $C(T,s)$ such that for all $t\leq T$, $||\nabla\Psi(t, \cdot)||_{H^s(\mathbb{R}_+^3)} \leq C(T,s)$. To finish the proof, it remains to show uniqueness and regularity in time.  Uniqueness follows from the usual energy method.  Indeed, let $\Psi_1$, $\Psi_2$ be two solutions with the same initial data $\nabla\Psi_0 \in H^s(\mathbb{R}_+^3)$ for some $s\geq3$. We will use the formulation of \cref{Pgradient} with $\tilde{\Psi}=\Psi_1-\Psi_2$, $\tilde{F}=F_1-F_2$.   Considering the difference of the two equations, we have 
$$ \partial_t(\nabla\tilde{\Psi}) + \grad^\perp \tilde{\Psi} \cdot \grad (\nabla\Psi_1) + \grad^\perp \Psi_2 \cdot \grad (\nabla\tilde{\Psi}) = \nabla\tilde{F} . $$
Multiplying by $\nabla\tilde{\Psi}$, using the regularity of $\nabla\Psi_1$, and integrating by parts, we have
\begin{align*}
\frac{1}{2}\frac{\partial}{\partial t}||\nabla\tilde{\Psi}||_{L^2}^2 &= \int_{\mathbb{R}_+^3} \left(\grad^\perp \tilde{\Psi} \cdot \grad (\nabla\Psi_1) + \grad^\perp \Psi_2 \cdot \grad (\nabla\tilde{\Psi})\right)\cdot\nabla\tilde{\Psi}  + \int_{\mathbb{R}_+^3} \nabla\tilde{F} \cdot \nabla\tilde{\Psi}\\
&= \int_{\mathbb{R}_+^3} \left(\grad^\perp \tilde{\Psi} \cdot \grad (\nabla\Psi_1)\right) \cdot \nabla\tilde{\Psi} + \int_{\mathbb{R}^2} \lap \tilde{\Psi} \tilde{\Psi} \\
&\leq C ||\nabla\tilde{\Psi}||_{L^2}^2 .
\end{align*}
Since $\nabla\tilde{\Psi}|_{t=0}=0$, Gronwall's inequality shows that $\nabla\tilde{\Psi}=0$ for all time.  
For the regularity in space and time, now assume that $\Psi$ is a solution to $(\QG)$ with smooth initial data.  Using the equalities 
$$ \partial_t (\Delta\Psi) = -\grad^\perp\Psi\cdot\grad(\Delta\Psi) $$
$$ \partial_t(\dnu) = -\grad^\perp\Psi\cdot\grad(\dnu)+\lap\Psi $$
and noticing that \cref{full} gives that any spatial derivative of the right hand side in either equality is bounded, we have that $\Delta\Psi$, $\dnu$ and all their spatial derivatives are $C^1$ in time.  Differentiating the equations in time and continuing inductively finishes the proof of \cref{main}. 
\end{proof}

\section{Appendix}

We now provide proofs of \cref{duality} and \cref{inequality} from the preliminaries.  
\begin{customthm}{2.5}\label{2.5}
\begin{enumerate}
\item Suppose that $(-\lap)^{-\frac{1}{4}} w \in L^\infty(\mathbb{R}^2)$ and $z \in \Hdot^{\frac{1}{2}}\cap L^\infty(\mathbb{R}^2)$ is supported in $B_2(0)$.  Then there exists $C$ independent of $w,z$ such that $$|| wz ||_{H^{-2}(\mathbb{R}^2)} \leq C ||(-\lap)^{-\frac{1}{4}} w||_{L^\infty(\mathbb{R}^2)}  \left( ||z||_{L^\infty(\mathbb{R}^2)} + ||z||_{\Hdot^\frac{1}{2}(\mathbb{R}^2)} \right) $$
\item Suppose that $z \in L^\infty \cap \Hdot^\frac{1}{2}(\mathbb{R}^2)$.  Then there exists $C$ independent of $z$ such that $$||z \half z||_{H^{-2}(\mathbb{R}^2)} \leq C\left( ||z||_{L^\infty(\mathbb{R}^2)}||z||_{\Hdot^\frac{1}{2}(\mathbb{R}^2)}+ ||z||^2_{\Hdot^\frac{1}{2}(\mathbb{R}^2)}\right)$$.  
\end{enumerate}
\end{customthm}

\begin{proof}
\begin{enumerate}
\item Suppose that $g\in H^2(\mathbb{R}^2)$. We first show that $\quarter(zg) \in L^1(\mathbb{R}^2)$.  By the compact support of $z$, we have that
\begin{align*}
| \quarter (zg)(x) \mathcal{X}_{\{|x|>3\}}(x) | &= \bigg{|} \mathcal{X}_{\{|x|>3\}}(x) P.V. \int_{\mathbb{R}^2} \frac{z(x)g(x) - z(y)g(y)}{|x-y|^{\frac{5}{2}}} \,dy \bigg{|}\\
&=  \bigg{|} \mathcal{X}_{\{|x|>3\}}(x) P.V. \int_{B_2(0)} \frac{- z(y)g(y)}{|x-y|^{\frac{5}{2}}} \,dy \bigg{|}\\ 
&\leq C \mathcal{X}_{\{|x|>3\}}(x) \int_{B_2(0)} \frac{ ||z||_{L^\infty} ||g||_{L^\infty} }{|x|^{\frac{5}{2}}}\,dy \\
&\leq C \mathcal{X}_{\{|x|>3\}}(x) \frac{ ||z||_{L^\infty} ||g||_{L^\infty} }{|x|^{\frac{5}{2}}}
\end{align*}
Integrating in $x$ then gives that 
$$ ||\quarter (zg)(x) \mathcal{X}_{\{|x|>3\}}(x)||_{L^1(\mathbb{R}^2)} \leq C ||z||_{L^\infty}||g||_{L^\infty}  $$
In addition, it follows from H\"{o}lder's inequality and a short calculation with the Gagliardo seminorm that 
\begin{align*}
||\quarter (zg)(x) \mathcal{X}_{\{|x|\leq 3\}}(x) ||_{L^1(\mathbb{R}^2)} &\leq C ||\quarter (zg)(x) ||_{L^2(\mathbb{R}^2)}\\
&\leq C \left( ||z||_{L^\infty}||g||_{\Hdot^\frac{1}{2}} + ||g||_{L^\infty}||z||_{\Hdot^\frac{1}{2}} \right)\\
\end{align*}
Then 
\begin{align*}
\int_{\mathbb{R}^2} w z g &= \int_{\mathbb{R}^2} (-\lap)^{-\frac{1}{4}} (w) \quarter (z g) \\
&\leq ||(-\lap)^{-\frac{1}{4}} w ||_{L^\infty(\mathbb{R}^2)} ||\quarter(zg)||_{L^1(\mathbb{R}^2)}\\
&\leq C||(-\lap)^{-\frac{1}{4}} w ||_{L^\infty} \left( ||z||_{L^\infty}||g||_{\Hdot^\frac{1}{2}} + ||g||_{L^\infty}|| z||_{\Hdot^\frac{1}{2}} + ||z||_{L^\infty}||g||_{L^\infty}\right)\\
&\leq C ||(-\lap)^{-\frac{1}{4}} w||_{L^\infty}  \left( ||z||_{L^\infty} + ||z||_{\Hdot^\frac{1}{2}} \right)||g||_{H^2}\\
\end{align*}
\item Suppose again that $g\in H^2(\mathbb{R}^2)$.  Then 
\begin{align*}
\int_{\mathbb{R}^2}  \half z(x) z(x) g(x)\,dx &= \int_{\mathbb{R}^2} (-\lap)^{\frac{1}{4}} (z)(x) \quarter (z g)(x)\,dx \\
&\leq ||(-\lap)^{\frac{1}{4}} z ||_{L^2(\mathbb{R}^2)} ||\quarter(zg)||_{L^2(\mathbb{R}^2)}\\
&\leq C|| z ||_{\Hdot^{\frac{1}{2}}(\mathbb{R}^2)} \left( ||z||_{L^\infty(\mathbb{R}^2)}||g||_{\Hdot^\frac{1}{2}(\mathbb{R}^2)} + ||g||_{L^\infty(\mathbb{R}^2)}|| z||_{\Hdot^\frac{1}{2}(\mathbb{R}^2)} \right)\\
\end{align*}
and the result follows from applying Sobolev embedding to $g$.  
\end{enumerate}
\end{proof}

We prove \cref{inequality}, following the proof of Proposition 2.104 in \cite{bcd}.

\begin{customthm}{2.6}\label{2.6}
There exists a constant $C$ such that for any $h = \grad H:\mathbb{R}^2\rightarrow\mathbb{R}^2$,
$$ ||h||_{L^\infty} \leq C||H||_{L^\infty} + {C}||h||_{\Bdot_{\infty,\infty}^0}\left( 1+\log{\frac{||h||_{\Hdot^\frac{3}{2}}}{||h||_{\Bdot_{\infty,\infty}^0}}} \right).$$
\end{customthm}

\begin{proof}
Let us set $\Theta(x)= 1 - \sum_{j=0}^\infty \Phi_j(x)$ where $\Phi_j$ is the function associated to the $j^{th}$ Littlewood-Paley projection.  Notice that since $\hat{\Theta}(\xi)$ is compactly supported, we have that 
$$ ||\Theta\ast h ||_{L^\infty} = ||\Theta\ast \grad H ||_{L\infty} = || \grad\Theta \ast H ||_{L^\infty} \leq C||H||_{L^\infty} $$
In addition, we have that by the characterizations of Besov spaces and Sobolev embedding, for $\epsilon=\frac{1}{2}$,
$$ \sup_{j\geq 0} 2^{j\epsilon}||\Delta_j h||_{L^\infty} \leq C||h||_{\Cdot^\epsilon} \leq C||h||_{\Hdot^\frac{3}{2}} . $$
We therefore have that 
\begin{align*}
||h||_{L^\infty} &= \left|\left| \Theta \ast h + \sum_{j=0}^\infty \Delta_j h\right|\right|_{L^\infty}\\
       &\leq ||\Theta \ast h||_{L^\infty} + \sum_{j=0}^{N-1} ||\Delta_j h||_{L^\infty} + \sum_{j=N}^\infty 2^{j\epsilon}||\Delta_j h||_{L^\infty} 2^{-j\epsilon} \\
       &\leq C||H||_{L^\infty} + N||h||_{\Bdot_{\infty,\infty}^0} + C||h||_{\Hdot^\frac{3}{2}}\frac{2^{-(N-1)\epsilon}}{2^\epsilon-1}\\
\end{align*}
and taking 
$$ N=1+\left(\frac{1}{\epsilon}\log_2{\frac{||h||_{\Hdot^\frac{3}{2}}}{||h||_{\Bdot_{\infty,\infty}^0}}} \right) $$
 finishes the proof.
\end{proof}

Let us now use the above proof to provide a short justification of the construction of the bump functions $\gamma_k$ in \cref{fine}.  Let $\gamma_k$ be a smooth bump function compactly supported in $B_{\frac{1}{2}+2^{-k-1}}$, equal to $\frac{1}{2}+2^{-k-1}$ on $B_{\frac{1}{2}+2^{-k-2}}$, and with $||\grad \gamma_k||_{L^\infty} \leq C2^k$. It is clear that the above argument works also for $h=\half H$, and $H=\gamma_k$.  Then using that $$ ||\half \gamma_k||_{\Bdot^0_{\infty,\infty}} \leq || \grad \gamma_k ||_{\Bdot^0_{\infty,\infty}} \leq || \grad \gamma_k ||_{L^\infty} \leq C2^k$$ and $$||\half\gamma_k||_{{\Bdot^\epsilon_{\infty,\infty}}} \leq ||\grad\gamma_k||_{{\Bdot^\epsilon_{\infty,\infty}}} \leq || \grad \gamma_k ||_{C^\epsilon} \leq C2^{(1+\epsilon)k},$$
 we have 
\begin{align*}
|| \half \gamma_k ||_{L^\infty} &\leq ||\gamma_k||_{L^\infty} + C||\half \gamma_k ||_{\Bdot^0_{\infty,\infty}} \left( 1+\log{\frac{||\half\gamma_k||_{\Bdot^\epsilon_{\infty,\infty}}}{||\half \gamma_k||_{\Bdot^0_{\infty,\infty}}}} \right) \\ 
&\leq 1+ C2^k  \left( 1+ \log{C2^{k(1+\epsilon)}}\right) \\
&\leq Ck2^k . 
\end{align*}

\bibliography{Paper1}
\bibliographystyle{plain}
\nocite{bb}
\nocite{bkm}
\nocite{cv}
\nocite{cw}
\nocite{cvicol}
\nocite{dg}
\nocite{knv}
\nocite{km}
\nocite{v}
\nocite{pv}
\nocite{kn}
\nocite{cmt}
\nocite{Bourgain2000}
\nocite{bcd}
\nocite{Chemin}
\nocite{Grafakos}
\nocite{Constantin2006}
\nocite{MR1312238}
\nocite{Pedlosky1987}
\nocite{ccv}
\nocite{cc}
\nocite{w}
\nocite{cmz}
\nocite{a}
\nocite{pavlovic}
\nocite{Kato}
\nocite{cg}
\nocite{Schwab2016}
\nocite{hitchhiker}

\end{document}